\documentclass[11pt]{amsart}
\usepackage{geometry}                
\usepackage{graphicx}
\usepackage{amsfonts,amsthm,amsmath,amssymb,latexsym, amscd, euscript}
\usepackage{graphicx,color}
\usepackage[all]{xy}
\usepackage{epsfig}
\usepackage{labelfig}
\usepackage{mathrsfs}

\usepackage{amssymb}
\usepackage{epstopdf}
\theoremstyle{plain}
\newtheorem{thm}{Theorem}[section]

\newtheorem{prop}[thm]{Proposition}
\newtheorem{lem}[thm]{Lemma}

\newtheorem{rem}[thm]{Remark}
\newtheorem{exam}[thm]{Example}

\newtheorem{definition}[thm]{Definition}

\title[The Heights Theorem]{The Heights Theorem for infinite Riemann surfaces}

\thanks{The author was supported in part by a PSC-CUNY grant, and a Simons foundation grant.}

\author{Dragomir \v Sari\' c}

\begin{document}
\maketitle

\begin{abstract} Marden and Strebel established the Heights Theorem for integrable holomorphic quadratic differentials on parabolic Riemann surfaces. We extends the validity of the Heights Theorem to all surfaces whose fundamental group is of the first kind. In fact, we establish a more general result: the {\it horizontal} map which assigns to each integrable holomorphic quadratic differential a measured lamination obtained by straightening the horizontal trajectories of the quadratic differential is injective for an arbitrary Riemann surface with a conformal hyperbolic metric. This was established by Strebel in the case of the unit disk. 

When a hyperbolic surface has a bounded geodesic pants decomposition, the horizontal map assigns a bounded measured lamination to each integrable holomorphic quadratic differential. When surface has a sequence of closed geodesics whose lengths go to zero, then there exists an integrable holomorphic quadratic differential whose horizontal measured lamination is not bounded. We also give a sufficient condition for the  non-integrable holomorphic quadratic differential to give rise to bounded measured laminations.
 \end{abstract}


\section{Introduction}

Let $X$ be an arbitrary Riemann surface equipped with a conformal hyperbolic metric. Denote by $A(X)$ the space of all integrable holomorphic quadratic differentials on $X$. Given a non-zero $\varphi\in A(X)$, denote by $\mathcal{V}_{\varphi}$ the horizontal foliation of $\varphi$. 
Let $\tilde{\varphi}$ and $\tilde{\mathcal{V}}_{\tilde{\varphi}}$ be lifts of $\varphi$ and $\mathcal{V}_{\varphi}$ to the unit disk $\Delta$ which is identified with the universal covering $\tilde{X}$ of $X$.
Marden and Strebel \cite{MardenStrebel}, \cite{Strebel} proved that all leaves of $\tilde{\mathcal{V}}_{\tilde{\varphi}}$ have two endpoints on the ideal boundary $\partial_{\infty}\Delta\equiv S^1$ of $\Delta$. We replace each leaf of $\tilde{\mathcal{V}}_{\tilde{\varphi}}$ by a hyperbolic geodesic with the same endpoints on $S^1$. In this fashion we obtain a geodesic lamination $\tilde{\lambda}_{\tilde{\varphi}}$ on $\Delta$ which projects to a geodesic lamination $\lambda_{\varphi}$ on $X$. We assign a measure to $\tilde{\lambda}_{\tilde{\varphi}}$ by the push-forward of the transverse measure of $\tilde{\mathcal{V}}_{\tilde{\varphi}}$  and thus obtain a measured lamination $\tilde{\nu}_{\tilde{\varphi}}$ on $\Delta$. The measured lamination $\tilde{\nu}_{\tilde{\varphi}}$ projects to a measured lamination $\nu_{\varphi}$ on $X$ corresponding to an integrable holomorphic quadratic differential $\varphi$, called a {\it horizontal measured lamination}. 

Denote by $ML(X)$ the space of all measured laminations on $X$. 
We define the {\it horizontal measure map}
$$
\mathcal{F}:A(X)\to ML(X)
$$
 by
$$
\mathcal{F}(\varphi )=\nu_{\varphi}.
$$
In the case of closed surfaces with genus at least two this map is proved to be a homeomorphism by Hubbard and Masur \cite{HubbardMasur}, and by Kerckhoff \cite{Kerckhoff}. For the case when $X=\Delta$ this map is proved to be injective by Strebel \cite{Strebel2}. Also for $X=\Delta$
a modular measure map (which is closely related to the horizontal measure map) is injective on the space of projective integrable holomorphic quadratic differentials with image  inside the space of projective bounded measured laminations $PML_b(\Delta)$ but the map is not onto (see \cite{HakobyanSaric}). See \cite{Thurston}, \cite{Saric}, \cite{Saric1} or Section 2 for the definition of bounded measured laminations.

We prove (see Theorem \ref{thm:vertical_inj})

\begin{thm}
\label{thm:main}
Let $X$ be an arbitrary Riemann surface equipped with a conformal hyperbolic metric. Then the horizontal measure map
$$
\mathcal{F}:A(X)\to ML(X),\ \ \ \ \mathcal{F}(\varphi )=\nu_{\varphi}
$$
is injective. 
\end{thm}

There is a recent interest in studying linear flows on infinite translation surfaces(for example, see \cite{CGL}, \cite{H}, \cite{H1}, \cite{HT}, \cite{HW}, \cite{HubertS}, \cite{HubertW}, \cite{PSV}, \cite{Trevino}, \cite{V}). The space of holomorphic quadratic differentials contains squares of Abelian differentials arising from the infinite translations surface. Linear flows on infinite translation surfaces of finite area are of a particular interest (see \cite{HT}, \cite{H1}, \cite{Trevino}).

Assume now that $X$ has a non-trivial fundamental group $\pi_1(X)$ and fix a non-zero $\varphi\in A(X)$.
For a simple closed geodesic $\gamma$ on $X$, the {\it height} of the homotopy class $[\gamma ]$ of $\gamma$ is defined by
$$
h_{\varphi}([\gamma ] ):=\inf_{\gamma'} \int_{\gamma'} d\mathcal{V}_{\varphi},
$$
where the infimum is over all differentiable simple closed curves $\gamma' \in [\gamma ]$. 

Therefore we have a heights map $h_{\varphi}:\mathcal{S}\to\mathbb{R}$ for every $\varphi$, where $\mathcal{S}$  is the set of all simple closed geodesic on $X$.
The Heights Theorem of Marden and Strebel (see \cite{MardenStrebel1} and \cite{Strebel}) states that the map which assigns the height function $h_{\varphi}$ to each $\varphi\in A(X)$ is injective if $X$ is a parabolic Riemann surface (see \cite{AhlforsSario} for the definition of a parabolic Riemann surface). The statement of the Heights Theorem is false when surface $X$ is not equal to its convex core-i.e., when it contains a funnel or a geodesic half-plane (see \cite{BS}), because a holomorphic quadratic differential may have horizontal trajectories which do not  essentially intersect any homotopy class of simple closed geodesics. The largest class of Riemann surfaces for which the Heights Theorem could possibly apply are the surfaces whose fundamental groups are of the first kind-i.e., surface that are equal to their convex cores.

Indeed, we first prove that a measured lamination on a Riemann surface (with the fundamental group of the first kind) is uniquely determined by its intersections with all simple closed geodesics (see Theorem \ref{thm:intersection-unique}). Then we have (see Theorem \ref{thm:heights})

\begin{thm}
Let $X$ be a Riemann surface whose fundamental group is of the first kind and $\varphi ,\varphi_1$ two integrable holomorphic quadratic differentials on $X$. If
$$
h_{\varphi}([\gamma ] )=h_{\varphi_1}([\gamma ] )
$$
for all $\gamma\in\mathcal{S}$ then
$$
\varphi =\varphi_1.
$$
\end{thm}

The image of the horizontal measure map
$$
\mathcal{F}:A(X)\to ML(X)
$$
is not surjective unlike in the case of compact surfaces. 
We first restrict our attention to surfaces that have bounded pants decomposition or bounded geometry or finite topology (see Theorems \ref{thm:integrable_measured_laminations}, \ref{thm:bounded_geom} and \ref{thm:funnels}). Denote by $ML_b(X)$ the space of all bounded measured laminations (for the definition, see \cite{Thurston}, \cite{Saric} or Section 2).

\begin{thm}
Let $X$ be a Riemann surface that either has a bounded pants decomposition or it is of bounded geometry.
Given an integrable holomorphic quadratic differential $\varphi$ on $X$, denote by $\nu_{\varphi}$ the horizontal measured lamination obtained by straightening the leaves of the horizontal foliation of $\varphi$. Then
$$
\nu_{\varphi}\in ML_b(X) .
$$
Moreover, if $X$ is a finite Riemann surface with a conformal hyperbolic metric of infinite area and $\varphi\in A(X)$ then
$$
\nu_{\varphi}\in ML_b(X).
$$
\end{thm}

In the case when $X$ has an upper bounded geodesic pants decomposition with boundary geodesics $\{\alpha_n\}_n$ and a subsequence $\alpha_{n_i}$ with $l_X(\alpha_{n_i})\to 0$ as $i\to\infty$,  we obtain that the image of $A(X)$ is not completely contained in $ML_b(X)$ (see Theorem \ref{thm:not_bounded} and Proposition \ref{thm:upper_bounded_measured}).  

\begin{thm}
Let $X$ be a Riemann surface with an upper bounded geodesic pants decomposition with boundary geodesics $\{\alpha_n\}_n$ and a subsequence $\{\alpha_{n_i}\}$ whose lengths go to zero. Then there exists $\varphi_0\in A(X)$ whose horizontal foliation has ring domains homotopic to $\alpha_{n_i}$  such that
$$
\nu_{\varphi_0}\in ML(X)- ML_b(X),\ \ \nu_{\varphi_0}(\alpha_{n_i})\to\infty .
$$
Moreover, for any $\varphi\in A(X)$ we have
$$
i(\nu_{\varphi},\alpha_{n_i})\leq\frac{C}{\sqrt{l_X(\alpha_{n_i})}}
$$
\end{thm}

The horizontal measure map is even more complicated when we consider the image of non-integrable holomorphic quadratic differentials. In order to introduce some control, we find  sufficient conditions on a Riemann surface $X$ and a non-integrable holomorphic quadratic differential $\varphi$ such that $\nu_{\varphi}$ is a bounded measured lamination (see Theorem \ref{thm:nonint}).

\begin{thm}
Let $X$ be an infinite hyperbolic surface equipped with two bounded geodesic pants decomposition $\mathcal{P}=\{ P_k\}$ and $\mathcal{P}'=\{ P_k\}'$ which do not share a boundary geodesic. Let $\{\alpha_n\}$, $\{\alpha_n'\}$ be the geodesic boundaries of $P_k$ and $P_k'$. Let $\mathcal{C}_n$ and $\mathcal{C}_n'$ be the standard collars of $\alpha_n$ and $\alpha_n'$. If 
$$
\sup_n\int_{\mathcal{C}_n}|\varphi (\zeta )|d\xi d\eta<\infty 
$$
and
$$
\sup_n\int_{\mathcal{C}_n'}|\varphi (\zeta )|d\xi d\eta<\infty 
$$
then 
$
\nu_{\varphi} 
$ is a bounded measured lamination. 
\end{thm}

\section{Riemann surfaces and measured laminations}

A Riemann surface different from the Riemann sphere, the complex plane, the once punctured complex plane and the torus supports a unique conformal metric of constant curvature $-1$, called the {\it hyperbolic metric}. Equivalently, if a Riemann surface is different from the above four surfaces then its universal covering $\tilde{X}$ is holomorphic to the unit disk $\Delta$. In this case $X$ is holomorphic to $\Delta/\Gamma$, where $\Gamma$ is a Fuchsian group. The hyperbolic metric on $X$ is the projection of the hyperbolic metric $\frac{2|dz|}{1-|z|^2}$ on $\Delta$. Unless explicitly states otherwise, all the geodesics in this paper are with respect to the hyperbolic metric. 

A geodesic in $\Delta$ has two distinct endpoints on the ideal boundary $S^1=\partial_{\infty}\Delta$. Conversely, given two distinct points $a$ and $b$ on $S^1$ there is a unique geodesic whose endpoints are $a$ and $b$. The space $G(\Delta)$ of (unoriented) geodesics of $\Delta$ is identified with $(S^1\times S^1-\Delta )/\mathbb{Z}_2$, where $\Delta$ is the diagonal and the $\mathbb{Z}_2$-action send $(a,b)$ to $(b,a)$.

A {\it geodesic pair of pants} is a bordered hyperbolic surface homeomorphic to sphere minus three closed disks whose border consists of three simple closed geodesics with possibly one or two of them begin degenerated to a cusp.

A {\it geodesic lamination} $\lambda$ on $X$ is a closed subset of $X$ equipped with a foliation by complete, pairwise disjoint geodesics. A lift $\tilde{\lambda}$ of $\lambda$ to the universal covering $\Delta$ is a geodesic lamination of $\Delta$ invariant under the action of $\Gamma =\pi_1(X)$.  
The geodesic lamination $\tilde{\lambda}$ is a closed subset of $G(\Delta)$. 

A {\it measured lamination} $\tilde{\mu}$ on $\Delta$ is a Radon measure on the space of geodesics $G(\Delta)$ whose support is a geodesic lamination $\tilde{\lambda}$. If $\tilde{\mu}$ is invariant under the action of $\pi_1(X)$ then it projects to a measured lamination $\mu$ on $X$. If $I$ is a compact hyperbolic arc (possibly a closed geodesic) on $X$ then the {\it intersection number} $i(\mu ,I)$ is defined to be the total $\mu$-mass of the geodesics intersecting $I$. This can be defined by lifting $I$ to $\Delta$ and using the lifted measure $\tilde{\mu}$. Conversely, a measured lamination is completely determined by the intersection numbers with all transverse compact geodesic arcs (see \cite{Bonahon}). The space of measured laminations on $X$ is denoted by $ML(X)$. 

We define a {\it Thurston norm} of a measured lamination $\mu$ on a conformally hyperbolic Riemann surface $X$ by
$$
\|\mu\|_{Th}:=\sup_I i(\mu ,I),
$$
where the supremum is over all geodesic arcs of length $1$ (for example, see \cite{Saric}, \cite{Saric1}). The space of {\it bounded measured laminations} $ML_b(X)$ consists of all measured laminations with finite Thurston norm.

A Riemann surface is said to be {\it infinite} if its fundamental group is infinitely generated. An infinite Riemann surface necessarily supports a conformal hyperbolic metric. An infinite Riemann surface is obtained by isometrically gluing countably many geodesic pairs of pants along boundaries of equal lengths and by attaching at most countably many funnels and geodesic half-planes (see \cite{AR}, \cite{BS}).

\section{ Horizontal foliations of integrable holomorphic quadratic differentials}

Denote by $X$ an arbitrary Riemann surface equipped with a conformal hyperbolic metric. Let $\varphi$ be a non-zero  holomorphic quadratic differential on $X$ of finite norm-i.e., $\varphi (\zeta )$ is holomorphic in local coordinates $\zeta$, $\varphi (\zeta )d\zeta^2$ is invariant under coordinate change and $\int_X|\varphi (\zeta )|d\sigma (\zeta )<\infty$, where $d\sigma (\zeta )$ is the Euclidean area form in the local coordinate $\zeta$. Note that the condition $\varphi (\zeta )\neq 0$ is independent of a local chart and we call such a point {\it regular}. Otherwise we have $\varphi (\zeta )=0$ in every local chart and we call such a point {\it singular}. In fact, the order of the zero is independent of the local coordinate.

Fix $\zeta_0$ in a local chart of $X$. The map $\zeta\mapsto w:=\int_{\zeta_0}^{\zeta}\sqrt{\varphi (\xi )}d\xi$ is called a {\it natural parameter} for $\varphi$. When $\zeta_0$ is a regular point, the natural parameter is a conformal map of a neighborhood of $\zeta_0$ onto an open subset of the complex plane $\mathbb{C}$. When $\varphi$ has a zero of order $n\geq 1$ at $\zeta_0$, the natural parameter is given by $w:=\eta^{\frac{n+2}{2}}$ for a local coordinate $\eta =\eta (\zeta )$ of a neighborhood of $\zeta_0$. A neighborhood of $\zeta_0$ is divided into $n+2$ sectors and each sector is mapped by the natural parameter $w$ into a Euclidean half-disk with center the origin such that  $w(\zeta_0)=0$. The holomorphic quadratic differential $\varphi$ in the natural parameter $w$ around a regular point equals $dw^2$ and around a singular point of order $n$ equals $(\frac{n+1}{2})^2w^ndw^2$ (see \cite{Strebel}).  If $\zeta_1$ and $\zeta_2$ are two natural parameters, then $\zeta_1=\pm\zeta_2+const$.

A straight $\theta$-arc on $X$ is the pre-image of a straight Euclidean arc in the natural parameter $w$ which subtends angle $\theta$ with the positive direction of the real axis. When $\theta =0$ it is called a {\it horizontal } arc and when $\theta$ then it is called a {\it vertical} arc. Equivalently, a differentiable arc $\gamma$ is a straight $\theta$-arc if $\mathrm{arg}\ [\varphi (\gamma (t))\gamma'(t)^2]=\theta$. A differentiable arc $\gamma$ is a {\it horizontal} ({\it vertical}) arc if $\varphi (\gamma (t))\gamma'(t)^2>0$ ($\varphi (\gamma (t))\gamma'(t)^2<0$). A maximal horizontal (vertical) arc is called a horizontal (vertical) {\it trajectory}.

\subsection{Lifted holomorphic quadratic differentials on $\Delta$}

Denote by $\tilde{\varphi}$ the lift of $\varphi$ to the universal covering $\tilde{X}$. 
We fix a biholomorphic identification $\tilde{X}\equiv\Delta:=\{ z=x+iy:|z|<1\}$. The universal covering  $\Delta$ is equipped with the standard hyperbolic metric $\frac{2|dz|}{1-|z|^2}$. The lift $\tilde{\varphi}(z)$ is a holomorphic function of the variable $z\in\Delta$. 
Note that $\tilde{\varphi}$ is not of finite norm on $\Delta$ unless $X=\Delta$. 

Fix a point $z_0\in\Delta$ that is regular for $\tilde{\varphi}(z)$-i.e., $\tilde{\varphi}(z_0)\neq 0$. 
The map $z\mapsto w:=\int_{z_0}^z\sqrt{\tilde{\varphi}(\zeta )}d\zeta$  is called the {\it natural parameter} $w$ of $\tilde{\varphi}$ for $\Delta$. It is a locally conformal map away from the critical points of $\tilde{\varphi}(z)$. At a critical point has the same behavior as described on the Riemann surface $X$. A {\it horizontal} ({\it vertical}) arc on $\Delta$ is the pre-image of a Euclidean horizontal (vertical) arc in the natural parameter $w$ of $\tilde{\varphi}$. A horizontal (vertical) {\it trajectory} is a maximal horizontal (vertical) arc in $\Delta$.

Each direction of a horizontal (vertical) trajectory either accumulates to the ideal boundary $\partial_{\infty}\Delta\equiv S^1$ or to a zero of $\tilde{\varphi}$ (see \cite{Strebel}). 
A {\it regular} horizontal trajectory of $\tilde{\varphi}$ is a horizontal trajectory whose both directions accumulate to $S^1$. Marden and Strebel \cite{MardenStrebel}, \cite{Strebel} proved that all regular horizontal trajectories have exactly two accumulation points on $S^1$ accumulated by the two opposite horizontal rays. 
A horizontal trajectory is said to be  {\it critical} if at least one ray ends at a critical point of $\tilde{\varphi}$.
The set of critical horizontal trajectories is at most countable since $\tilde{\varphi}$ has at most countably many zeros in $\Delta$ and each zero has finitely many horizontal trajectories ending at it. We extend each critical horizontal trajectory at a zero of $\tilde{\varphi}$ by concatenating another critical horizontal trajectory. We continue in this fashion at each new zero of $\tilde{\varphi}$ that the concatenated trajectory meets. After an at most countably many concatenations, we obtain a simple curve consisting of an at most countably many critical horizontal trajectories whose both ends accumulate to $S^1$. We will call such curves {\it generalized} horizontal trajectories. It follows from Marden and Strebel \cite{MardenStrebel} that all generalized horizontal trajectories have two distinct accumulation points on $S^1$.

\subsection{From a horizontal foliation to a measured lamination}

The family of horizontal trajectories of $\tilde{\varphi}$ defines a singular foliation of $\Delta$ with singularities at the zeroes of $\tilde{\varphi}$. If $z$ is a zero of $\tilde{\varphi}$ of order $n$ then the foliation has an $(n+2)$-pronged singularity at $z$. 

Let $c$ be a compact rectifiable arc transverse to the foliation of $\tilde{\varphi}$. 
We define a measured foliation $\tilde{\mathcal{V}}_{\tilde{\varphi}}$ by assigning an intersection number $i(\tilde{\mathcal{V}}_{\tilde{\varphi}},c)$ of the foliation of $\tilde{\varphi}$ with any $c$,
$$
i(\tilde{\mathcal{V}}_{\tilde{\varphi}},c):=\int_c |Im(\sqrt{\tilde{\varphi}(z)}dz)|=\int_c |dv|,
$$
where the second integral is in the natural parameter $w=u+iv$ and the vertical displacement $dv$ corresponds to $Im(\sqrt{\tilde{\varphi}(z)}dz)$. 
The intersection number is invariant under homotopies that respect the leaves of the foliation and $\tilde{\mathcal{V}}_{\tilde{\varphi}}$ is called a {\it horizontal measured foliation} of $\tilde{\varphi}$.
 
 Let $z_1,z_2\in\Delta$. We define the {\it vertical distance} $h_{\tilde{\varphi}}(z_1,z_2)$ between $z_1$ and $z_2$ by
 $$
 h_{\tilde{\varphi}}(z_1,z_2)=\inf_c\int_c |Im(\sqrt{\tilde{\varphi}(z)}dz)|
 $$
 where the infimum is over all rectifiable curves $c$ that connect $z_1$ and $z_2$. Note that the vertical distance is a pseudo-metric as the vertical distance of two points on the same horizontal trajectory is zero.
 
Let $\beta$ be a vertical interval not containing a critical point of $\tilde{\varphi}$. The {\it horizontal strip} $\mathcal{S}(\beta )$ for $\beta$ is the union of all regular and critical horizontal trajectories that have a point in common with $\beta$. Note that $\mathcal{S}(\beta )$ does not contain critical points since critical horizontal trajectories end at critical points but they do not contain them.
 By \cite[Theorem 19.2, page 87]{Strebel}, there exists countably many open vertical intervals 
 $\{\beta_i\}_i$ such that the open strips $\{\mathcal{S}(\beta_i)\}_i$ are mutually disjoint and cover $\Delta$ with the exception of countably many critical horizontal trajectories and critical points of $\tilde{\varphi}$.

The horizontal measured foliation $\tilde{\mathcal{V}}_{\tilde{\varphi}}$ is the lift to $\Delta$ of the horizontal measured foliation $\mathcal{V}_{\varphi}$ of $\varphi$ on $X$.  Given a regular (or a generalized) horizontal trajectory $h$ of $\tilde{\varphi}$, denote by $G(h)$ the geodesic of $\Delta$ that has the same endpoints on $\partial_{\infty}\Delta=S^1$ as $h$. 
We assign to $\tilde{\mathcal{V}}_{\tilde{\varphi}}$ a geodesic lamination $\tilde{\lambda}_{\tilde{\varphi}}$ by replacing each regular (or generalized) horizontal trajectory $h$ of $\tilde{\varphi}$ with a hyperbolic geodesic $G(h)$ and taking the closure of the obtained set of geodesics in $G(\Delta)$. By construction, $\tilde{\lambda}_{\tilde{\varphi}}$ is invariant under the action of the fundamental group $\pi_1(X)=\Gamma$ and it projects to a geodesic lamination $\lambda_{\varphi}$ of $X$.

In order to define a measured lamination $\nu_{\varphi}$ whose support is $\lambda_{\varphi}$, we consider the lift $\tilde{\lambda}_{\tilde{\varphi}}$ of $\lambda_{\varphi}$ to $\Delta$.  If $g\in\tilde{\lambda}_{\tilde{\varphi}}$ is not a lift of a closed geodesic on $X$, then Marden and Strebel \cite[Theorem 2(c)]{MardenStrebel} proved that there exists a unique horizontal trajectory $h$ in $\tilde{\mathcal{V}}_{\tilde{\varphi}}$ such that $G(h)=g$. The horizontal trajectory $h$ is either regular or generalized. If $g\in\tilde{\lambda}_{\tilde{\varphi}}$ is a lift of a closed geodesic then the set of horizontal trajectories 
which have the same endpoints on $S^1$ as $g$ is a component of the lift of a ring domain for $\varphi$ on $X$. The two extreme horizontal trajectories of this set are generalized trajectories which are lifts of the boundary components of the ring domain on $X$ corresponding to the closed geodesic.

Let $I$ be a compact geodesic arc transverse to $\tilde{\lambda}_{\tilde{\varphi}}$. Denote by $\tilde{\lambda}_{\tilde{\varphi}}^I$ the set of geodesics in $\tilde{\lambda}_{\tilde{\varphi}}$ that intersect $I$. Let $g_1$ and $g_2$ be geodesics in $\tilde{\lambda}_{\tilde{\varphi}}^I$ such that any geodesic $g\in\tilde{\lambda}_{\tilde{\varphi}}^I$ is between $g_1$ and $g_2$.
If $g_i$, for $i=1,2$, is not a lift of a closed geodesic on $X$, denote by $h_i$ the unique horizontal trajectory such that $G(h_i)=g_i$. If $g_i$ is a lift of a closed geodesic on $X$, 
denote by $h_i$ one of the two horizontal trajectories on the boundary of the lift of the ring domain such that all horizontal trajectories that correspond to the geodesics of $\tilde{\lambda}_{\tilde{\varphi}}^I$ are on one side of $h_i$. By our construction, a horizontal trajectory gives a geodesic of $\tilde{\lambda}_{\tilde{\varphi}}^I$ if and only if it separates $h_1$ and $h_2$.

Let $\mathcal{S}(\beta_1)$ and $\mathcal{S}(\beta_2)$ be two horizontal strips from the fixed family $\{\mathcal{S}(\beta_i)\}_i$ that covers $\Delta$ which contain $h_1$ and $h_2$ in their interiors or on their boundaries. Let $\{\mathcal{S}(\beta_{i_j})\}_{j=1}^{j_0}$, $j_0\leq\infty$, be the remaining horizontal strips that consist of horizontal trajectories which separate $h_1$ and $h_2$. Let $b_{i_j}=\int_{\beta_{i_j}}|dv|$ be the height of the horizontal strip $\mathcal{S}(\beta_{i_j})$-i.e., the length of $\beta_{i_j}$. Let $b_1$ be the height of the part of the strip $\mathcal{S}(\beta_1)$ that consists of horizontal trajectories which separate $h_1$ and $h_2$. We use analogous definition for $b_2$ and $\mathcal{S}(\beta_2)$. Then we define
$$
i(\tilde{\nu}_{\tilde{\varphi}},I):=b_1+b_2+\sum_{j=1}^{j_0} b_{i_j}.
$$
Since we defined the intersection with any compact geodesic arc and the intersection is invariant under the homotopy which preserves the leaves of $\tilde{\lambda}_{\tilde{\varphi}}$, we obtained a measured lamination $\tilde{\nu}_{\tilde{\varphi}}$ whose support is the geodesic lamination $\tilde{\lambda}_{\tilde{\varphi}}$. The measured lamination $\tilde{\nu}_{\tilde{\varphi}}$ is invariant under the action of $\pi_1(X)$. Thus we obtained a measured lamination $\nu_{\varphi}$ on $X$, called a {\it horizontal measured lamination}, whose support is the geodesic lamination $\lambda_{\varphi}$.
Since the height distances are countably additive, it follows that $\nu_{\varphi}$ is independent of the choice of the covering of $\Delta$ by a countable family of horizontal strips.

A closed horizontal trajectory on a Riemann surface $X$ is a part of a maximal ring domain of closed homotopic horizontal trajectories. The boundary of the maximal ring domain contains critical points and critical trajectories (see \cite{Strebel}). A regular horizontal trajectory is called a {\it spiral} if both its ends accumulate to the trajectory itself. Fix an open oriented vertical arc $\beta$ and denote by $\beta^{+}$ and $\beta^{-}$ the two sides of $\beta$. The {\it spiral set} $\mathscr{S}(\beta )$ determined by $\beta$ is the set of horizontal spirals which intersect $\beta$. Each $\alpha\in \mathscr{S}(\beta )$ is divided into countably many components $\alpha_n$ by the points $\alpha\cap\beta$. Each component $\alpha_n$ either has its endpoints on the two sides of $\beta$ (called the first kind), or it has both endpoints on $\beta^{+}$ or on $\beta^{-}$ (called the second kind). The spiral set is partitioned into an at most countably many rectangular sets of components of the spirals of either the first or second kind. A rectangular set of the first kind is a measurable set of components of spirals inside a rectangle with one vertical side on one side of $\beta$ and  the other vertical side on the other side of $\beta$ and no other points in common with $\beta$. A rectangular set of the second kind is a measurable set of components of spirals inside a rectangle with both vertical sides on one side of $\beta$ and no other points in common with $\beta$ (see \cite[Section 13.6]{Strebel}). A horizontal trajectory is called a {\it cross-cut} if both of its ends accumulate to infinity of $X$. The cross-cut trajectories belong to an at most countable family of strips of horizontal cross-cuts (see \cite[Section 13.5]{Strebel}).
Finally, the definitions for  horizontal trajectories work work in the same manner for vertical trajectories,

When $\varphi$ is integrable, the set of horizontal trajectories consists of closed horizontal trajectories, recurrent horizontal trajectories, boundary horizontal trajectories and a set of zero $\varphi$-area (see \cite[Section 13]{Strebel}).

When $\pi_1(X)=\Gamma$ is a Fuchsian group of the first kind, the holomorphic quadratic differential $\varphi$ is allowed to be non-integrable. Marden and Strebel \cite{MardenStrebel} proved that each horizontal trajectory still has two different endpoints on $\partial_{\infty}\tilde{X}$ and we  define a corresponding geodesic lamination $\lambda_{\varphi}$ as above. The horizontal measured lamination $\nu_{\varphi}$ is well-defined.

We define the notion of a height of a homotopy class of a closed curve (see \cite{Strebel}, \cite{MardenStrebel1}).

\begin{definition}
Let $X$ be an infinite Riemann surface and $\gamma$ a homotopically non-trivial and boundary non-homotopic simple closed curve on $X$. The height of the homotopy class $[\gamma ]$ of $\gamma$ is given by
$$
h_{\varphi}([\gamma ] ):=\inf_{\gamma_1\sim\gamma}\int_{\gamma_1}dv,
$$
where the infimum is over all simple closed curves $\gamma_1$ homotopic to $\gamma$ and $u+iv$ is the natural parameter for $\varphi$.
\end{definition}

We will need a notion of a step curve.

\begin{definition}
A curve $s$ which is obtained by concatenation of finitely many horizontal and vertical intervals is called a {\it step curve}.
\end{definition}

The following lemma establishes the correspondence between the heights $h_{\varphi}$ of the homotopy classes of simple closed curves and the intersections of simple closed geodesics with $\nu_{\varphi}$.

\begin{lem}
\label{lem:heights-intersections}
Let  $\varphi$ be a holomorphic quadratic differential on a Riemann surface $X$. We assume that either $\varphi$ is integrable or the fundamental group of $X$ is of the first kind. If $\gamma$ is a simple closed geodesic on $X$ then
$$
h_{\varphi}([\gamma ] )= i(\nu_{\varphi},\gamma ).
$$
\end{lem}

\begin{proof}
Since $\gamma$ is a simple closed geodesic it follows that it is not null-homotopic or homotopic to a puncture.
We lift $\varphi$ to a holomorphic quadratic differential $\tilde{\varphi}$ on $\Delta$. Let $\tilde{\gamma}\subset\Delta$ be a single compact arc that covers $\gamma$ injectively except at its endpoints. In addition, we assume that two regular horizontal trajectories $h_1$ and $h_2$ go through the endpoints of $\tilde{\gamma}$. By \cite[Page 152, Theorem 24.1]{Strebel}, for every $\epsilon >0$ there is a step curve $s_{\epsilon}$ connecting $h_1$ and $h_2$ such that 
$$
h_{{\varphi}}([{\gamma} ])\geq \int_{s_{\epsilon}}|dv|-\epsilon .
$$
Recall that we defined $i(\tilde{\nu}_{\tilde{\varphi}}, \tilde{\gamma})$ to be the total height of the horizontal trajectories that separate $h_1$ and $h_2$. Since $s_{\epsilon}$ has to intersect each horizontal trajectory that separates $h_1$ and $h_2$ we have that $i(\tilde{\nu}_{\tilde{\varphi}}, \tilde{\gamma})\leq \int_{s_{\epsilon}}|dv|$. Combining the above two inequalities and letting $\epsilon\to 0$ gives
$$
h_{\varphi}(\gamma )\geq i(\nu_{\varphi},\beta )
$$
since $i(\tilde{\nu}_{\tilde{\varphi}},\tilde{\gamma})=i(\nu_{\varphi},\gamma )$.

\begin{figure}[h]
\leavevmode \SetLabels
\L(.48*.18) $h_1$\\
\L(.48*.74) $h_2$\\
\endSetLabels
\begin{center}
\AffixLabels{\centerline{\epsfig{file =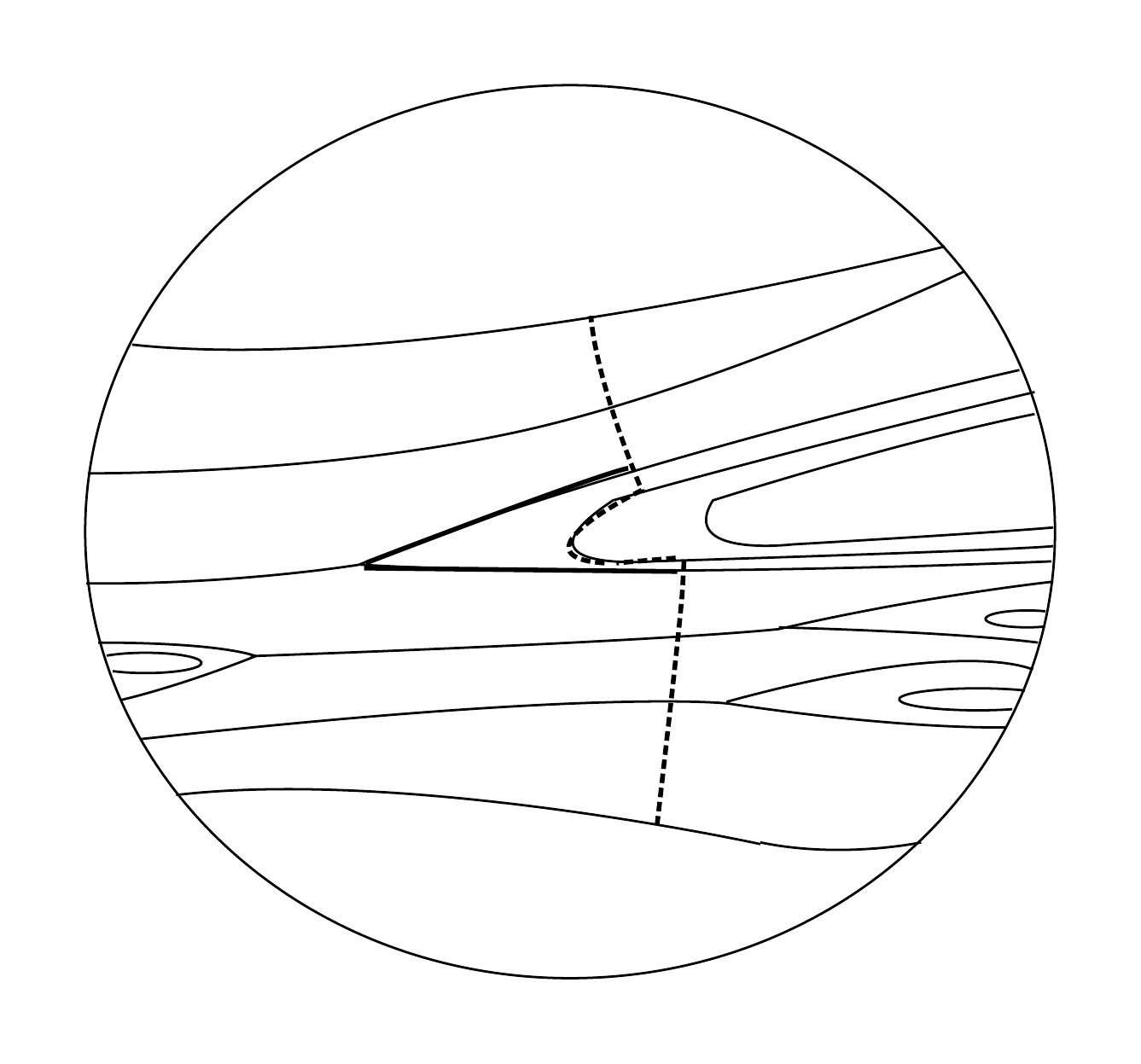,width=7.0cm,angle=0} }}
\vspace{-20pt}
\end{center}
\caption{A strip which does not separate $h_1$ and $h_2$. The dotted curve is $s_{\epsilon}$ and the bold curve is the replacement for the part of $s_{\epsilon}$ to obtain $s_{\epsilon}'$.}
\end{figure}

It remains to prove the opposite inequality. We consider the above step curve $s_{\epsilon}$ that satisfies $\int_{s_{\epsilon}}|dv|\geq h_{{\varphi}}([{\gamma}] )\geq \int_{s_{\epsilon}}|dv|-\epsilon$. We also recall that $i(\tilde{\nu}_{\tilde{\varphi}}, \tilde{\gamma})$ is the sum of the heights of the horizontal strips whose horizontal trajectories separate $h_1$ and $h_2$.

The step curve $s_{\epsilon}$ might cross a horizontal strip $\tilde{S}(\beta_j)$ whose horizontal trajectories do not separate $h_1$ and $h_2$. There are two cases to consider: either $\tilde{S}(\beta_j)$ does not separate $h_1$ and $h_2$, or  $\tilde{S}(\beta_j)$ separates $h_1$ and $h_2$ while each horizontal trajectory in $\tilde{S}(\beta_j)$ does not separate $h_1$ and $h_2$.
In the former case, one boundary component $\partial_1 \tilde{S}(\beta_j)$ separates  $\tilde{S}(\beta_j)$ from both $h_1$ and $h_2$. Then $s_{\epsilon}$ crosses $\partial_1 \tilde{S}(\beta_j)$ to enter $\tilde{S}(\beta_j)$ and then crosses it again to exit $ \tilde{S}(\beta_j)$ (see Figure 1). 
We modify $s_{\epsilon}$ by replacing the subcurve that enters $\tilde{S}(\beta_j)$ from $\partial_1 \tilde{S}(\beta_j)$ and then exists it from the same boundary component $\partial_1 \tilde{S}(\beta_j)$ by an arc along the boundary component $\partial_1 \tilde{S}(\beta_j)$ to obtain a step curve with a smaller height. We perform this modification for each horizontal strip of the given type and denote the modified curve by $s_{\epsilon}'$.

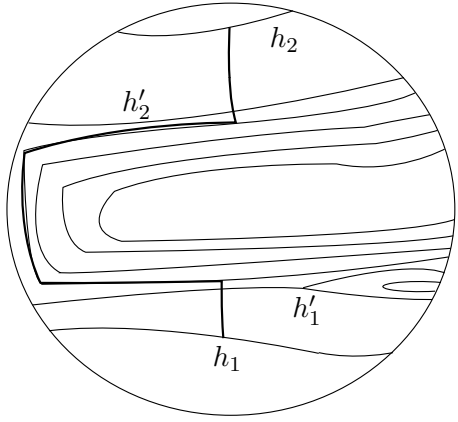
\begin{figure}[h]
\leavevmode \SetLabels
\L(.48*.18) $h_1$\\
\L(.55*.28) $h_1'$\\
\L(.4*.7) $h_2'$\\
\L(.53*.83) $h_2$\\
\endSetLabels
\begin{center}
\AffixLabels{\centerline{\epsfig{file =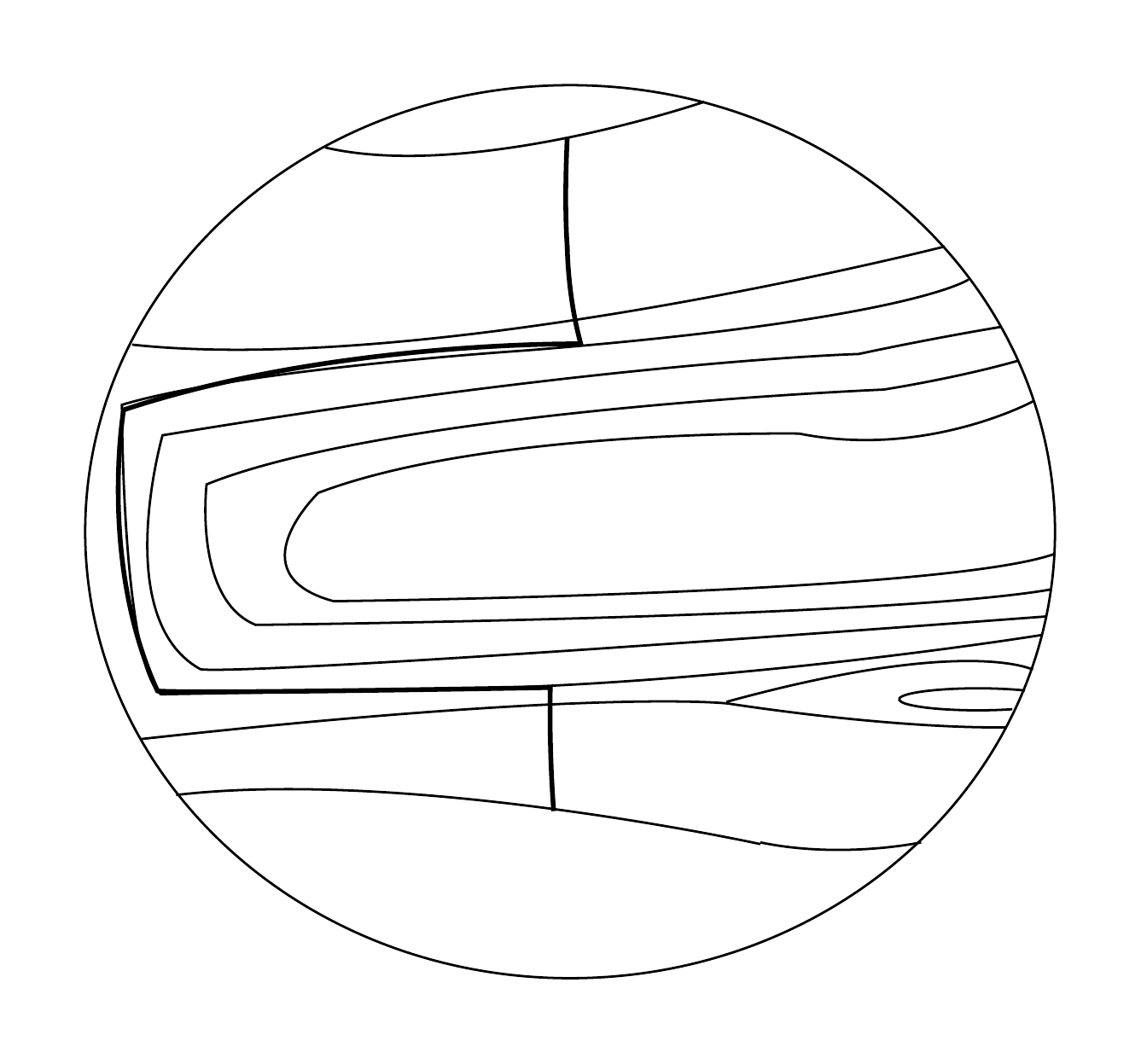,width=7.0cm,angle=0} }}
\vspace{-20pt}
\end{center}
\caption{A strip which separates $h_1$ and $h_2$ whose horizontal arcs do not separate $h_1$ and $h_2$.}
\end{figure}

We assume now that some strip that $s_{\epsilon}'$ crosses is in the later case. Namely $s_{\epsilon}'$ crosses a strip $\tilde{S}(\beta_j)$ that separates $h_1$ and $h_2$ while each horizontal trajectory in $\tilde{S}(\beta_j)$ does not separate $h_1$ and $h_2$ (see Figure 2). We modify $s_{\epsilon}'$ by replacing $s_{\epsilon}'\cap\mathcal{S}(\beta_j)$ with a step curve that starts by a short vertical arc from one boundary component followed by a long horizontal arc and then connected to the other boundary component by a short vertical arc (see Figure 2). We can arrange that the vertical arcs have arbitrary small heights. 
There are at most countably many strips of this kind and we arrange that the total height of all modified arcs is less than $\epsilon$. The new step curve is denoted by $s_{\epsilon}''$ and it has the property that the height of the part of $s_{\epsilon}''$ that crosses strips which consist of trajectories which do not separate $h_1$ and $h_2$ is at most $\epsilon$. The horizontal trajectories which do not separate $h_1$ and $h_2$ correspond to the geodesics of $\tilde{\lambda}_{\tilde{\varphi}}$ that do not cross $I$. We finally modify $s_{\epsilon}''$ such that the heights of the parts crossing the strips containing the  trajectories which separate $h_1$ and $h_2$ have the heights equal to the heights of the strips. Then we have
$$
h_{{\varphi}}(\beta )\leq \int_{s_{\epsilon}''}|dv|\leq i(\nu_{\varphi},\beta )+\epsilon .
$$
By letting $\epsilon \to 0$ we obtain $h_{{\varphi}}(\beta )\leq i(\nu_{\varphi},\beta )$ and the lemma is proved.
\end{proof}

The above lemma established that the height of a homotopy class of a simple closed curve is equal to the intersection number of the corresponding closed geodesic with the horizontal measured lamination.  We will need the corresponding statement for the intersection of a cross-cut vertical trajectory  with ${\mathcal{V}}_{{\varphi}}$ and the intersection number between the corresponding infinite geodesic on $X$ and the horizontal measured lamination ${\nu}_{{\varphi}}$.

\begin{lem}
\label{lem:cross-cut-height}
Let $X$ be a hyperbolic Riemann surface and $\varphi$ an integrable holomorphic quadratic differential on $X$. If $\gamma$ is a cross-cut vertical trajectory and $g=G(\gamma )$ is its corresponding hyperbolic geodesic then
$$
i({\mathcal{V}}_{{\varphi}}, \gamma )=i(\nu_{\varphi},g).
$$
\end{lem}

\begin{proof}
We lift $\gamma$ to $\tilde{\gamma}$ and $g$ to $\tilde{g}$ in the universal covering $\Delta$. Denote by $H_1$ and $H_2$ the two endpoints of $\tilde{g}$ on $S^1$. Let $h_1^n$ and $h_2^n$ be  horizontal trajectories (possibly singular) that are not lifts of closed horizontal trajectories belonging to an interior of a horizontal cylinder of $\varphi$. In addition, we assume that $h_i^{n}$ accumulate to $H_i$ as $n\to\infty$, and that $h_{i}^n$ separates $h_i^{n-1}$ and $h_i^{n+1}$. 

Consider the infimum of the heights of all arcs connecting $h_1^n$ and $h_2^n$ and denote it by $a^n_{\varphi}$.  Denote by $\tilde{\gamma}_n$ the subarc of $\tilde{\gamma}$ that connects $h_1^n$ and $h_2^n$. Since every horizontal arc that separates $h_1^n$ and $h_2^n$ must intersect $\tilde{\gamma}_n$ and every horizontal arc can intersect $\tilde{\gamma}$  at most once (by the uniqueness of the geodesic connection \cite[Page 75, Theorem 16.1]{Strebel}, it follows that
$$
a^n_{\varphi}=\int_{\tilde{\gamma}_n}|dv|.
$$

Let $g_n$ be a geodesic arc connecting the geodesics $G(h_1^n)$ and $G(h_2^n)$. 
The method of the proof of Lemma \ref{lem:heights-intersections} gives that 
$$i(\tilde{\mathcal{V}}_{\tilde{\varphi}},\tilde{\gamma}_n)=i(\tilde{\nu}_{\tilde{\varphi}},g_n).
$$
By the additivity we have $i(\tilde{\mathcal{V}}_{\tilde{\varphi}},\tilde{\gamma})=\lim_{n\to\infty}i(\tilde{\mathcal{V}}_{\tilde{\varphi}},\tilde{\gamma}_n)$ and $i(\tilde{\nu}_{\tilde{\varphi}},g)=\lim_{n\to\infty}i(\tilde{\nu}_{\tilde{\varphi}},g_n)$, and the lemma follows.
\end{proof}

\section{Intersection numbers and measured laminations on infinite surfaces}

A surface is said to be {\it infinite} if its fundamental group is infinitely generated. 
This section extends the classical theorem on the intersection numbers for closed surfaces to infinite surfaces whose fundamental group is of the first kind.

Let $X$ be an infinite Riemann surface that is equal to its convex core $\mathcal{C}(X)$. Equivalently, the action of the fundamental group $\Gamma =\pi_1(X)$ on the universal covering $\Delta=\tilde{X}$ is of the first kind (see \cite{BS}). 
Let $\mu$ be a measured lamination on $X$. For a simple closed geodesic $\gamma$, let $i(\mu ,\gamma )$ be the geometric intersection number of $\gamma$ with the measured lamination $\mu$. Equivalently, $i( \mu ,\gamma )$ is the transverse measure of $\mu$ deposited to a simple closed geodesic $\gamma$. We prove that if two measured laminations $\mu_1$ and $\mu_2$ have the same intersection numbers with each simple, closed geodesic $\gamma$ then they are equal.

\begin{thm}
\label{thm:intersection-unique}
Let $X$ be an infinite Riemann surface such that the action of its fundamental group on $\Delta$ is of the first kind. Let $\mu_1$ and $\mu_2$ be two measured laminations on $X$ such that 
$$
i( \mu_1,\gamma )=i( \mu_2, \gamma )
$$
for all simple closed geodesics $\gamma$ on $X$. Then
$$
\mu_1 =\mu_2.
$$
\end{thm}

\begin{proof}
Fix a topological pants decomposition of $X$ such that each boundary curve belongs to two different pairs of pants. By \cite{BS} (see also \cite{AR}), the topological pants decomposition can be  straightened to a geodesic pants decomposition of $X$.  Now we simply recall that the statement for closed surfaces is proved by considering the unions of two adjacent pairs of pants that make a four holed sphere (see \cite{FLP}, \cite{PH}). In order to prove that the measured laminations are the same it is enough to prove that the transverse measures (for $\mu_1$ and $\mu_2$) of the homotopy classes of arcs of geodesics connecting boundary sides of the geodesic pants decompositions are equal. This is achieved in the compact case by using the fact that the intersection numbers with the four boundary geodesics and three additional simple closed geodesics in the interior of the four holed sphere are equal. Since this fact holds for infinite surfaces as well we conclude that $\mu_1=\mu_2$.
\end{proof}

\section{The correspondence between integrable holomorphic quadratic differentials and measured laminations}

Let $X$ be a hyperbolic Riemann surface without any conditions on its topology or geometry. 
We define a {\it horizontal measure map} $\mathcal{F}$ from the space $A(X)$ of integrable holomorphic quadratic differentials on $X$ to the space $ML(X)$ of measured laminations on $X$ by straightening the horizontal foliations of quadratic differentials to measured laminations on $X$. Namely,
$$
\mathcal{F}:A(X)\to ML(X)
$$
given by
$$
\mathcal{F}(\varphi )=\nu_{\varphi}.
$$

We will need the following lemma.

\begin{lem}
\label{lem:crosscut-comp}
Let $X$ be an arbitrary hyperbolic Riemann surface and $\varphi,\varphi_1$ two integrable holomorphic quadratic differentials on $X$. If 
$$
\nu_{\varphi}=\nu_{\varphi_1}
$$
then for every cross-cut vertical trajectory $\gamma$ of $\varphi$ we have
$$
\int_{{\gamma}}|Im(\sqrt{\varphi (z)}dz)|\leq\int_{{\gamma}}|Im(\sqrt{\varphi_1 (z)}dz)|.
$$
\end{lem}

\begin{proof}
Fix an identification of the unit disk $\Delta$ with the universal cover of $X$. Lift $\varphi ,\varphi_1$ to holomorphic quadratic differentials $\tilde{\varphi},\tilde{\varphi}_1$ on $\Delta$. Note that $\tilde{\varphi},\tilde{\varphi}_1$ are not integrable on $\Delta$. However, the trajectory structure has the same properties as for integrable holomorphic quadratic differentials on $\Delta$ (see \cite{MardenStrebel}, \cite[Page 93, Section 19.7]{Strebel}). In particular, each trajectory has well-defined two endpoints on the ideal boundary $S^1$, and if a vertical and horizontal trajectory ray start at the same point of $\Delta$ their endpoints on $S^1$ are different. 

For a single lift $\tilde{\gamma}$ of a cross-cut vertical trajectory $\gamma$ of $\varphi$ we have $\int_{{\gamma}}Im(\sqrt{\varphi (z)}dz)= \int_{\tilde{\gamma}}Im(\sqrt{\tilde{\varphi} (z)}dz)$ and $\int_{{\gamma}}Im(\sqrt{\varphi_1 (z)}dz)= \int_{\tilde{\gamma}}Im(\sqrt{\tilde{\varphi}_1 (z)}dz)$ since the covering $\tilde{\gamma}\mapsto \gamma$ is injective. 
Let $\tilde{g}:=G(\tilde{\gamma})$ be the hyperbolic geodesic that shares the endpoints with $\tilde{\gamma}$. Let  $u+iv=\int_{*}\sqrt{\tilde{\varphi} (z)}dz$ be the natural parameter for $\tilde{\varphi}$ and $u_1+idv_1$ the natural parameter for $\tilde{\varphi}_1$ on $\Delta$.

By Lemma \ref{lem:cross-cut-height} we  have
$$
i( \tilde{\nu}_{\tilde{\varphi}},\tilde{g})=\int_{\tilde{\gamma}}dv.
$$

A geodesic of the support of $\tilde{\nu}_{\tilde{\varphi}_1}$ intersects $\tilde{g}$ if and only if the corresponding vertical trajectory intersects $\tilde{\gamma}$. This follows from the fact that all the endpoints are different and the existence of the intersection is decided based on the separation of the endpoints on $S^1$. 
Then 
$$
i(\tilde{\nu}_{\tilde{\varphi}_1},\tilde{g})\leq \int_{\tilde{\gamma}}|dv_1|,
$$
where a strict inequality is possible since $\tilde{\gamma}$ may intersect a horizontal trajectory of $\tilde{\varphi}_1$ more than once or it may intersect horizontal trajectories of $\tilde{\varphi}_1$ whose endpoints are not separated by the endpoints of $\tilde{\gamma}$. 

Our assumption that $\nu_{{\varphi}}=\nu_{{\varphi}_1}$ implies that
$$
i(\tilde{\nu}_{\tilde{\varphi}},\tilde{g})=i(\tilde{\nu}_{\tilde{\varphi}_1},\tilde{g})
$$
and we obtain
$$
\int_{\tilde{\gamma}}|dv|\leq\int_{\tilde{\gamma}}|dv_1|.
$$
\end{proof}

We are ready to prove the injectivity of the horizontal measure map. 

\begin{thm}
\label{thm:vertical_inj}
The horizontal measure map
$$
\mathcal{F}:A(X)\to ML(X)
$$
is injective.
\end{thm}

\begin{proof}
Assume that $\varphi,\varphi_1\in A(X)$ such that $\mathcal{F}(\varphi )=\mathcal{F}(\varphi_1)$. We need to prove that $\varphi=\varphi_1$. 

We first fix the vertical foliation of $\varphi$ on $X$. Recall that up to measure zero the surface $X$ is divided into countably many horizontal cylinders, spiral sets and horizontal strips of cross-cuts (see \cite{Strebel}). 

If $R_k$ is either a vertical cylinder or a vertical spiral set, Marden and Strebel (see \cite{MardenStrebel1}, \cite[Sections 24.2, 24.5]{Strebel}) proved that 
\begin{equation}
\label{eq:cylinder-spiral}
\iint_{R_k}|\varphi|dxdy\leq \iint_{R_k}|\varphi_1|dxdy
\end{equation}
under the condition that $h_{\varphi}(\gamma )=h_{\varphi_1}(\gamma )$ for every simple closed 
curve $\gamma$ on $X$ that cannot be homotoped to a point or boundary component. Our assumption that $\nu_{\varphi}=\nu_{\varphi_1}$ implies this condition by Lemma \ref{lem:heights-intersections}. 

The set of vertical cross-cuts is divided into countably many mutually disjoint(except for the boundary trajectories) measurable vertical strips (see \cite{Strebel}) and let $S_j$ be one of the strips induced by a measurable subset $\beta_j$ of a horizontal arc. We use the natural parameter $w=u+iv$ for $\varphi$ to represent the strip $S_j$ in the complex plane with $\beta_j$ a measurable subset of the horizontal segment $[0,c_j]$ on the $u$-axis.

For $u\in \beta_j$, denote by $l_{\varphi}(u)$ the $dv$-measure of the vertical trajectory $\tilde{t}_u=\{ u\}\times [a_u,b_u]$ with the real part $u$. By Lemma \ref{lem:crosscut-comp} we have
$$
l_{\varphi}(u)=\int_{\tilde{t}_u}dv\leq\int_{\tilde{t}_u}|dv_1|=\int_{a_u}^{b_u}\Big{|}\frac{\partial v_1}{\partial v}\Big{|}dv
$$ 
and by integrating in $du$ we get
$$
\iint_{S_j}|\varphi (z)|dxdy=\int_{\beta_j}\int_{a_u}^{b_u}dvdu\leq \int_{\beta_j}\int_{a_u}^{b_u}\Big{|}\frac{\partial v_1}{\partial v}\Big{|}dvdu.
$$
By applying the Cauchy-Schwarz inequality to the above we obtain
$$
\Big{(}\iint_{S_j}|\varphi (z)|dxdy\Big{)}^2\leq \Big{(}\iint_{S_j}\Big{|}\frac{\partial v_1}{\partial v}\Big{|}dudv\Big{)}^2\leq \Big{(}\iint_{S_j}dudv\Big{)}\Big{(}\iint_{S_j}\Big{|}\frac{\partial v_1}{\partial v}\Big{|}^2dudv\Big{)}.
$$
The above inequality together with
$$
\iint_{S_j}\Big{|}\frac{\partial v_1}{\partial v}\Big{|}^2dudv\leq \iint_{S_j}\Big{(}\Big{|}\frac{\partial v_1}{\partial v}\Big{|}^2+\Big{|}\frac{\partial v_1}{\partial u}\Big{|}^2\Big{)}dudv=\iint_{S_j}|\varphi_1 (z)|dxdy
$$ 
implies that
\begin{equation}
\label{eq:crosscuts}
\iint_{S_j}{|}\varphi (z){|}dxdy\leq\iint_{S_j}{|}\varphi_1 (z){|}dxdy.
\end{equation}

By summing the inequalities (\ref{eq:cylinder-spiral}) and (\ref{eq:crosscuts}) over all cylinders, spiral sets and cross-cuts of the vertical foliation of $\varphi$ we obtain
$$
\iint_X|\varphi (z)|dxdy\leq\iint_X|\varphi_1(z)|dxdy
$$
and since the argument is symmetric we obtain the reverse inequality as well. Thus we obtain
\begin{equation}
\label{eq:equal_norms}
\iint_X|\varphi (z)|dxdy=\iint_X|\varphi_1 (z)|dxdy
\end{equation}
The equality (\ref{eq:equal_norms}) together with the method of the proof for (\ref{eq:crosscuts}) and (\ref{eq:cylinder-spiral}) implies that $\partial v_1/\partial u\equiv 0$ and the equality in the Cauchy-Schwarz inequality implies that $|\partial v_1/\partial v|\equiv 1$. Thus $v_1\equiv \pm v+const$ on $X$ and we conclude that $\varphi =\varphi_1$.
\end{proof}

Marden and Strebel \cite{MardenStrebel1} proved the Heights Theorem on the class of parabolic surfaces. Namely, for $X$ parabolic the Heights Theorem states that the map which to each integrable holomorphic quadratic differential $\varphi$ assigns the heights  $h_{\varphi}:\mathcal{S}\to\mathbb{R}$ is injective, where $\mathcal{S}$ is the set of all simple closed geodesics on $X$. Both closed and finite area hyperbolic Riemann surfaces are parabolic and a large number of infinite surfaces are also parabolic (see \cite{AhlforsSario}, \cite{BasHakSar}). On the other hand, a Cantor tree surface with geodesic pants decomposition whose boundary geodesics have lengths bounded between two positive constants (or a complement of the standard  middle-third Cantor set) is not parabolic  (see \cite{McM}). 
If a surface has a funnel end or if it contains a hyperbolic half-plane then the Heights Theorem does not apply since a differential $\varphi$ can have horizontal trajectories that do not intersect any simple closed geodesics.
The largest class of Riemann surfaces to which the Heights Theorem might hold are those whose covering group is of the first kind. Indeed, we prove

\begin{thm}
\label{thm:heights}
Let $X$ be an infinite Riemann surface whose fundamental group is of the first kind and let $\varphi,\psi\in A(X)$. If
$$
h_{\varphi}([\gamma ] )=h_{\psi}([\gamma ] )
$$
for all $\gamma\in\mathcal{S}$ then 
$$
\varphi =\psi .
$$
\end{thm}

\begin{proof}
Since the fundamental group of $X$ is of the first kind and since $h_{\varphi}(\beta )=i(\beta ,\nu_{\varphi})$ and $h_{\psi}(\beta )=i(\beta ,\nu_{\psi})$, Theorem \ref{thm:intersection-unique} implies that
$$
\nu_{\varphi}=\nu_{\psi} .
$$
Then Theorem \ref{thm:vertical_inj} gives the conclusion.
\end{proof}

If $X$ is of a parabolic type, then the family of cross-cut vertical trajectories of any integrable holomorphic quadratic differential on $X$ has zero area. 
To see that Lemma \ref{lem:crosscut-comp} is necessary for the proof of Theorem \ref{thm:vertical_inj}, 
we give examples of integrable holomorphic quadratic differentials whose vertical foliations have strips of cross-cuts with positive area.

\begin{exam}
Let $X$ be a Riemann surface with a conformal hyperbolic metric that has a funnel end. Let $\alpha$ be the ideal boundary curve facing the funnel and let $\hat{X}$ be the double of $X$ across $\alpha$. Fix a simple closed curve $\gamma$ on $\hat{X}$ that essentially intersects $\alpha$. Strebel established that there exists an integrable holomorphic quadratic differential $\varphi_{\gamma}$ with one ring domain homotopic to $\gamma$. By restricting $\varphi_{\gamma}$ to $X$ we obtain an integrable holomorphic quadratic differential whose a.e. vertical trajectory is a cross-cut both ends accumulating to the ideal boundary curve $\alpha$.
\end{exam}

\begin{exam}
Let $X$ be a Riemann surface with a conformal hyperbolic metric that contains a hyperbolic half-plane which is attached along a bi-infinite geodesic to the convex core (see \cite{Bas}, \cite{BS}). We double $X$ along the ideal arc on the boundary of the hyperbolic plane and repeat the construction in the previous example to obtain an integrable holomorphic quadratic differential whose a.e. vertical trajectory is a cross-cut.
\end{exam}

\begin{exam}
Let $X$ be the unit disk $\Delta$ minus a countable set of points that accumulate to each point of the ideal boundary $S^1$ and are not equal to $0\in\Delta$. In this case the Riemann surface $X$ does not contain funnels or hyperbolic half-planes (see \cite{BS}). Equivalently, the covering group of $X$ is of the first kind. The holomorphic quadratic differential $dz^2$ is integrable on $X$ and a.e. vertical trajectory has both ends accumulating to $S^1$. It follows that the set of vertical cross-cuts has non-zero area. In particular, $X$ is not of a parabolic type.
\end{exam}

\section{Integrable holomorphic quadratic differentials on surfaces with bounded pants decompositions}

Let $X$ be an infinite hyperbolic surface equipped with a geodesic pants decomposition $\mathcal{P}=\{ P_i\}$. Each pair of pants $P_i$ has three boundary components which can be either simple closed geodesics or punctures, and at least one is a simple closed geodesic. Denote by $\{\alpha_n\}$ the set of all boundary geodesics of the pants $\{ P_i\}$. 
The pants decomposition $\mathcal{P}=\{ P_i\}$ is said to be {\it bounded} if there is a constant $M$ such that, for all $n$,
$$
1/M\leq l_X(\alpha_n)\leq M
$$
where $l_X(\cdot )$ is the length in the hyperbolic metric of $X$. If a pair of pants has a puncture on the boundary then the above condition does not apply to that boundary component. 

A hyperbolic surface obtained by gluing infinitely many geodesic pairs of pants with boundary geodesics bounded between two positive constants such that no boundary geodesic is left unglued to another boundary geodesic of a pair of pants is necessarily complete. Indeed, completeness follows because an upper bound on $l_X(\alpha_n)$ implies that every boundary geodesic $\alpha_n$ has a collar of definite width (see \cite{Buser}). Since a path leaving every compact subset of $X$ has to either go to a puncture or cross infinitely many such collars, it follows that it has an infinite length. Thus $X$ is complete and $\pi_1(X)$ is of the first kind.

Let $\alpha$ be a simple closed geodesic on $X$ and $\mathcal{C}_{\alpha}$ be its standard collar. Namely,
$$
\mathcal{C}_{\alpha}:=\{ \zeta\in X|d_X(\zeta ,\alpha )\leq\sinh^{-1}\frac{1}{\sinh ( l_X(\alpha )/2)}\}
$$
where $d_X$ is the hyperbolic distance on $X$ and $l_X(\alpha )$ is the hyperbolic length of $\alpha$ (see \cite{Buser}). 

We assume that $h_{\varphi}([\alpha ])\neq 0$, i.e. $\mathcal{V}_{\varphi}$ has no closed non-critical horizontal  trajectory homotopic to $\alpha$. 
Let $\mathcal{B}_{\alpha}$ be the set of horizontal arcs for ${\varphi}$ that connect the two boundaries $\alpha_1$ and $\alpha_2$ of the collar $\mathcal{C}_{\alpha}$, and that do not pass through a zero of $\varphi$ in $\mathcal{C}_{\alpha}$. 

We will need the following lemma.

\begin{lem}
\label{lem:restriction}
With the above notation, there exists a closed curve $\alpha'$ in $\mathcal{C}_{\alpha}$ homotopic to $\alpha$ such that
$$
i(\mathcal{V}_{\varphi},\alpha')=i(\mathcal{B}_{\alpha},\alpha').
$$
In other words, the total transverse measure $\mathcal{V}_{\varphi}$ of $\alpha'$ is equal to that of $\mathcal{B}_{\alpha}$.
\end{lem}

\begin{proof}
Let $\alpha_1$ be one boundary component of $\mathcal{C}_{\alpha}$. We modify $\alpha_1$ to a homotopic closed curve $\alpha'$ such that 
$$
i(\mathcal{V}_{\varphi},\alpha')=i(\mathcal{B}_{\alpha},\alpha').
$$
Consider the set $d_{\alpha}$ of points of $\alpha_1$ that do not lie on a leaf of $\mathcal{B}_{\alpha}$. 
A point $P$ of $d_{\alpha}$ can either lie on a singular leaf of $\mathcal{B}_{\alpha}$, or it is a zero of $\varphi$, or it lies on a leaf of $\mathcal{V}_{\varphi}$ that is tangent to $\alpha_1$ but does not enter $\mathcal{C}_{\alpha}$ near that point, or  it lies on a leaf of $\mathcal{V}_{\varphi}$ that connects $\alpha_1$ to itself inside $\mathcal{C}_{\alpha}$. The number of points in the first two cases is at most finite because the closure of $\mathcal{C}_{\alpha}$ contains at most finitely many zeros of $\varphi$. In the third case, all nearby horizontal trajectories connect $\alpha_1$ to itself and separate the point $P$ from $\alpha_2$. Thus we have at most countably many points in the third case. Since we are computing the $dv$-integrals 
 it is enough to consider the fourth case.

Given a horizontal arc $h'$ for ${\varphi}$ that connects $\alpha_1$ to itself inside $\mathcal{C}_{\alpha}$, the interval of $\alpha_1$ separated  from $\alpha_2$ by $h'$ is a subset of $d_{\alpha}$. Thus every component of $d_{\alpha}$ (except possibly finitely many corresponding to the first three cases) has an interior and $d_{\alpha}$ has at most countably many components. Let $K$ be a closed interval of $\alpha_1$ that is the closure of a component to $d_{\alpha}$
and denote by $x$ and $y$ its endpoints. The horizontal trajectories $h_x$ and $h_y$ for $\varphi$ that contain $x$ and $y$ necessarily intersect the interior of $\mathcal{C}_{\alpha}$. 
Let $h'_x$ and $h'_y$ be subarcs of $h_x$ and $h_y$ that connect from the inside the two boundaries of $\mathcal{C}_{\alpha}$, and have $x$ and $y$ as their endpoints. There are two different possibilities for $h_x'$ and $h_y'$. 

One possibility is that both $h_x'$ and $h_y'$ connect $\alpha_1$ to $\alpha_2$. We modify $\alpha_1$ by replacing $K$ 
with a path $p_K$ starting at $x$ following $h_x'$ then following an arc of $\alpha_2$ from the endpoint of $h_x'$ 
to the endpoint of $h_y'$ then following $h_y'$ to end at $y$ (see Figure 3). The path $p_K$ is built by arcs of the leaves of $\mathcal{V}_{\varphi}$ and consequently
$$
i(\mathcal{V}_{\varphi},p_K)=0.
$$

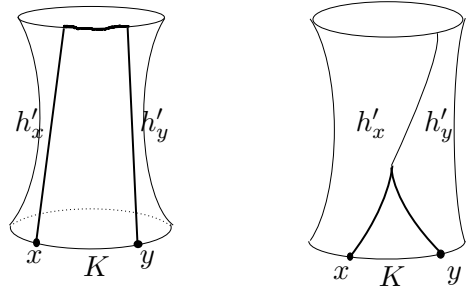
\begin{figure}[h]
\leavevmode \SetLabels
\L(.31*.1) $x$\\
\L(.41*.1) $y$\\
\L(.36*.06) $K$\\
\L(.58*.05) $x$\\
\L(.68*.05) $y$\\
\L(.62*.03) $K$\\
\L(.3*.5) $h_x'$\\
\L(.41*.5) $h_y'$\\
\L(.6*.5) $h_x'$\\
\L(.66*.5) $h_y'$\\
\endSetLabels
\begin{center}
\AffixLabels{\centerline{\epsfig{file =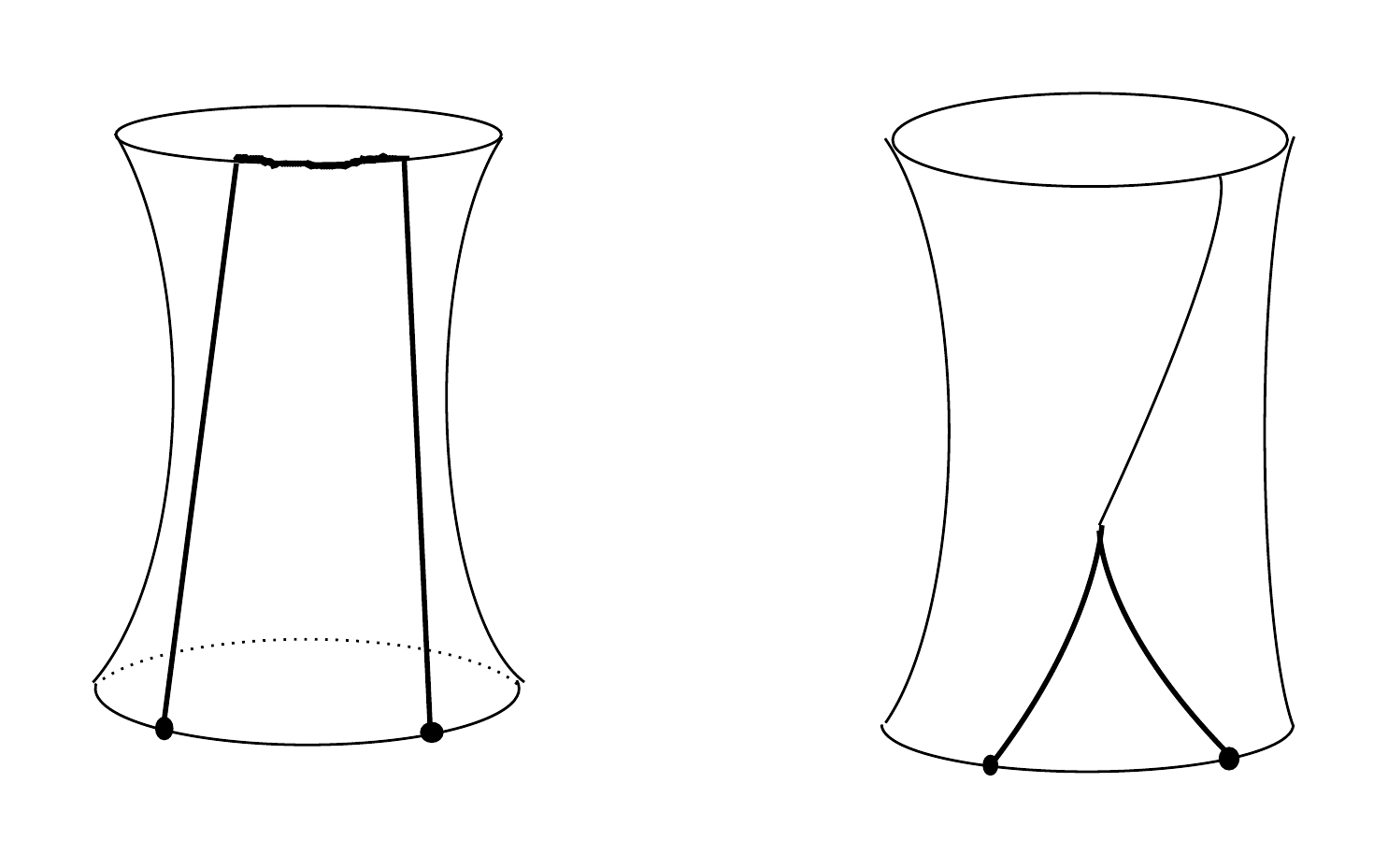,width=7.0cm,angle=0} }}
\vspace{-20pt}
\end{center}
\caption{The curve $\alpha'$. }
\end{figure}

The other possibility is that $h_x'$ and $h_y'$ meet at a zero of $\varphi$ inside $\mathcal{C}_{\alpha}$ (see Figure 3). Then we replace $K$ by an arc $p_K$ which consists of the parts of $h_x'$ and $h_y'$ from $x$ and $y$ to the zero of $\varphi$. Note that
$$
i(\mathcal{V}_{\varphi},p_K)=0.
$$

We perform this modification in an at most countably many disjoint (except possibly at endpoints) subarcs of $\alpha_1$ to obtain a close curve $\alpha'$ which is clearly homotopic to $\alpha_1$ and hence to $\alpha$ as well. By the above property of the intersection numbers, we obtain
$$
i(\mathcal{V}_{\varphi},\alpha')=i(\mathcal{B}_{\alpha},\alpha').
$$
\end{proof}

We are ready to prove the main theorem of this section. 

\begin{thm}
\label{thm:integrable_measured_laminations}
Let $\varphi$ be an integrable holomorphic quadratic differential on an infinite hyperbolic surface $X$ equipped with a bounded geodesic pants decomposition. Let $\nu_{\varphi}$ be the horizontal measured lamination obtained by straightening the leaves of the horizontal foliation of $\varphi$ as above. Then
$$
\|\nu_{\varphi}\|_{Th}<\infty .
$$
\end{thm}

\begin{proof}
Let $\{\alpha_n\}_n$ be the family of boundary geodesics of a bounded geodesic pants decomposition of $X$. 
Since $\varphi$ is integrable, there is at most one vertical ray ending at each puncture of $X$ (see Strebel \cite[Page 31]{Strebel}). Therefore the geodesics of the support of $\nu_{\varphi}$ do not have an endpoint at a puncture of $X$. 

Let $\mathcal{V}_{\varphi}$ be the horizontal foliation of $\varphi$. Denote by $\mathcal{C}_n$ the standard collar around $\alpha_n$.
Let $\mathcal{V}_n$ be  the set of all horizontal arcs for $\varphi$ in $\mathcal{C}_n$ that connect its two boundary sides. We note that in general $\mathcal{V}_n$ does not cover the standard collar $\mathcal{C}_n$ as leaves of $\mathcal{V}_{\varphi}$ may connect one boundary of $\mathcal{C}_n$ to itself before crossing the other boundary. 
Define
$$
a_n:=h_{\varphi}([\alpha_n]),
$$
to be the height of the homotopy class of the curve $\alpha_n$.

Let
$a_n'$ be the total transverse measure of $\mathcal{V}_n$. Namely, we choose at most countable collection of arcs $\{T_n^k\}_k$ transverse to $\mathcal{V}_n$ such that 
the set of arcs of $\mathcal{V}_n$ that intersect both $T_n^k$ and $T_n^{k_1}$ consists of at most one arc.
We set
$$
a_n':=\sum_k\int_{T_n^k}|Im (\sqrt{\varphi (z)}dz)|=\sum_k\int_{(T_n^k) '}|dv|,
$$
where $(T_n^k)'$ is the image of $T_n^k$ in the natural parameter $w=u+iv$ for $\varphi$.

Since $a_n$ is the infimum of the intersections of $\mathcal{V}_{\varphi}$ with all the curves in the homotopy class $[\alpha_n]$ of $\alpha_n$, Lemma \ref{lem:restriction}  gives
$$
a_n\leq a_n'.
$$

Let $\Gamma_n$ be the set of all curves in the standard collar $\mathcal{C}_n$ around $\alpha_n$ connecting the two boundary sides. Since $\mathcal{V}_n\subset \Gamma_n$ and $1/M\leq l_X(\alpha_n)\leq M$ for all $n$, we have
\begin{equation}
\label{eq:transverse1}
C'\geq \mathrm{mod}(\Gamma_n)\geq \mathrm{mod}(\mathcal{V}_n)=\sum_k\int_{(T_n^k)'}\frac{1}{l_n(w)}dv
\end{equation}
where $(T_n^k)'$ is a transverse arc to $\mathcal{V}_n$, $l_n(w)$ is the Euclidean length of the leaf of $\mathcal{V}_n$ through $w$, $w=u+iv$ is the natural parameter for $\varphi$ and $dv$ is the transverse measure to $\mathcal{V}_{\varphi}$ in the natural parameter. The first inequality follows because $\mathrm{mod}(\Gamma_n)$ and $ l_X(\alpha_n)$ are comparable on compact subsets of $(0,\infty )$ (for example, see Maskit \cite{Maskit}). The equality is obtained by showing that $\frac{1}{l_n(w)}dv$ is the extremal metric for $\mathcal{V}_n$ which is a consequence of Beurling's criteria (see \cite[Lemma 2.8]{HakobyanSaric}).

On the other hand, the area computation for the standard collar gives
\begin{equation}
\label{eq:transverse2}
\|\varphi\|_{L^1(X)}=\int_X|\varphi (z)|dxdy\geq \sum_k\int_{(T_n^k)'}l_n(w)dv=\int_{\cup_k (T_n^k)'}l_n(w)dv.
\end{equation}
To simplify the notation we define $T_n':= \cup_k (T_n^k)'$. 
Using the Cauchy-Schwarz inequality together with (\ref{eq:transverse1}), (\ref{eq:transverse2}) and Lemma \ref{lem:heights-intersections} we obtain
\begin{equation}
\label{eq:intersection_cuff}
i(\nu_{\varphi},\alpha_n)^2\leq \Big{(}\int_{T_n'}dv\Big{)}^2\leq\int_{T_n'}l_n(w)dv\int_{T_n'}\frac{1}{l_n(w)}dv\leq C'\|\varphi\|_{L^1(X)}
\end{equation}
for all $n$, where $i(\nu_{\varphi},\alpha_n)$ is the geometric intersection number between a measured lamination $\nu_{\varphi}$ and a simple closed geodesic $\alpha_n$. 

Let $b_n$ be the height of the cylinder of $\varphi$ in the homotopy class of $\alpha_n$ if it exists. Let $M_n$ be the module of this cylinder. Since $l_X(\alpha_n)\geq 1/M$ for all $n$, it follows that $M_n\leq M'$ for all $n$, where $M'$ depends only on $M$ (see \cite{Maskit}).
Then
\begin{equation}
\label{eq:height}
\frac{1}{M'}\sum_nb_n^2\leq \sum_n\frac{b_n^2}{M_n}\leq\|\varphi\|_{L^1(X)}
\end{equation}
which implies that 
\begin{equation}
\label{eq:atom}
\nu_{\varphi}(\alpha_n)\leq M'\|\varphi\|_{L^1(X)},
\end{equation} where $\nu_{\varphi}(\alpha_n)$ is the atomic part of $\nu_{\varphi}$ on $\alpha_n$. 

We are ready to prove that $\nu_{\varphi}$ has a bounded Thurston norm. Let $I$ be a geodesic arc on $X$ of length $1$. Then $I$ intersects at most $k$ geodesic pairs of pants of the fixed bounded pants decomposition, where $k$ depends on the bound $M$. Any geodesic of the support of $\nu_{\varphi}$ that intersects $I$ is either a boundary geodesic belonging to the finitely many pairs of pants that $I$ intersects or it intersects at least one of the boundary geodesics of the finitely many pairs of pants. The total $\nu_{\varphi}$-mass of finitely many boundary geodesics and geodesics intersecting finitely many boundary geodesics is bounded by a constant multiple of the number of boundary geodesics by (\ref{eq:intersection_cuff}) and (\ref{eq:atom}). Therefore $i(\nu_{\varphi},I)\leq M''$ for all geodesic arcs $I$ on $X$ of length $1$ which finishes the proof. 
\end{proof}





\section{Integrable holomorphic quadratic differentials on surfaces with bounded geometry}

In this section we consider a class of surfaces introduced by Kinjo \cite{Kinjo} that contains as a proper subclass all surfaces equipped with a bounded geodesic pants decomposition. 

Let $X$ be an infinite hyperbolic surface and $S$ a subsurface whose boundary consists of simple closed geodesics. Assume that all the components of $X-S$ are planar surfaces. If there exists $M>0$ such that $S$ has a geodesic pants decomposition whose lengths are between $1/M$ and $M$, and the distance of each $z\in X-S$ to the boundary of $X-S$ is bounded above by $M$ then $X$ is said to be of {\it bounded geometry} (see \cite{Kinjo}). 

We remark that the bound on the pants decomposition of $S$ implies the bound on the boundary geodesics of $S$ which are also boundary geodesics of $X-S$. Kinjo \cite{Kinjo} proved that $X$ can be decomposed into right-angled hexagons with an upper bound $M’$ on the side lengths.

\begin{thm}
\label{thm:bounded_geom}
Let $X$ be an infinite hyperbolic surface with bounded geometry and $\varphi$ an integrable holomorphic quadratic differential on $X$. Then
$$
\|\nu_{\varphi}\|_{Th}<\infty .
$$
\end{thm}

\begin{proof}
Kinjo \cite{Kinjo1} proved that $X-S$ can be broken into hexagons with an upper bound on the side lengths. The hexagons are constructed using common orthogonals to the closed geodesics bounding the planar part $X-S$ that does not have a bounded pants decomposition. The orthogonals in the planar part also have a lower bound on their lengths since they have to cross half of a standard collar with geodesics of bounded lengths. Therefore the hexagons have bounded sides both from the above and below.

Let $\{\alpha_n\}$ be the boundary geodesics of the fixed pants decomposition of $S$. 
The intersection number $i(\nu_{\varphi},\alpha_n)$ is uniformly bounded by the same method as in the proof of Theorem \ref{thm:integrable_measured_laminations} since $\alpha_n$ have bounded lengths. By the same reason, we have that $\nu_{\varphi}(\alpha_n)$ is uniformly bounded as well. 

Denote by $\{\alpha_{n_k}\}$ the family of all geodesics that are on the boundary of $X-S$ and 
note that it is a subfamily of $\{\alpha_n\}$. Since $l_X(\alpha_{n_k})$ is bounded from above and the side lengths of hexagons are bounded from below, there is an upper bound $N$ on the number of hexagons meeting each boundary geodesic  $\alpha_{n_k}$. Let $a$ be one side of a hexagon that is orthogonal to boundary geodesics $\alpha_{n_1}$ and $\alpha_{n_2}$ of $S$. We form a closed curve $\gamma_a$  as follows. Let $P_{n_i}$ be the pair of pants of the pants decomposition of $S$ that has $\alpha_{n_i}$ on its boundary. For the point $A_{n_i}:=a\cap \alpha_{n_i}$, we choose a simple loop $b_i$ inside $P_{n_i}$ based at $A_{n_i}$  that separates the other two boundary components of $P_{n_i}$. Then $\gamma_a$ is obtained by concatenating $a$, $b_1$, $a$ and $b_2$ in the given order. The closed curve $\gamma_a$ is  homotopic to a simple  and homotopically non-trivial curve whose length is bounded above. Therefore a simple closed geodesic $\gamma_a^{*}$ homotopic to $\gamma_a$  also has a uniformly bounded length for all choices of the orthogonal $a$ (see Figure 4).
The proof of Theorem \ref{thm:integrable_measured_laminations} gives that $i(\gamma_a^{*},\nu_{\varphi})$ is uniformly bounded for all $a$. 

\begin{figure}[h]
\leavevmode \SetLabels
\L(.4*.05) $\alpha_{n_2}$\\
\L(.7*.78) $\alpha_{n_1}$\\
\L(.41*.17) $b_2$\\
\L(.56*.74) $b_1$\\
\L(.44*.49) $A_{n_2}$\\
\L(.55*.53) $A_{n_1}$\\
\L(.51*.48) $a$\\
\L(.35*.2) $P_{n_2}$\\
\L(.62*.8) $P_{n_1}$\\
\endSetLabels
\begin{center}
\AffixLabels{\centerline{\epsfig{file =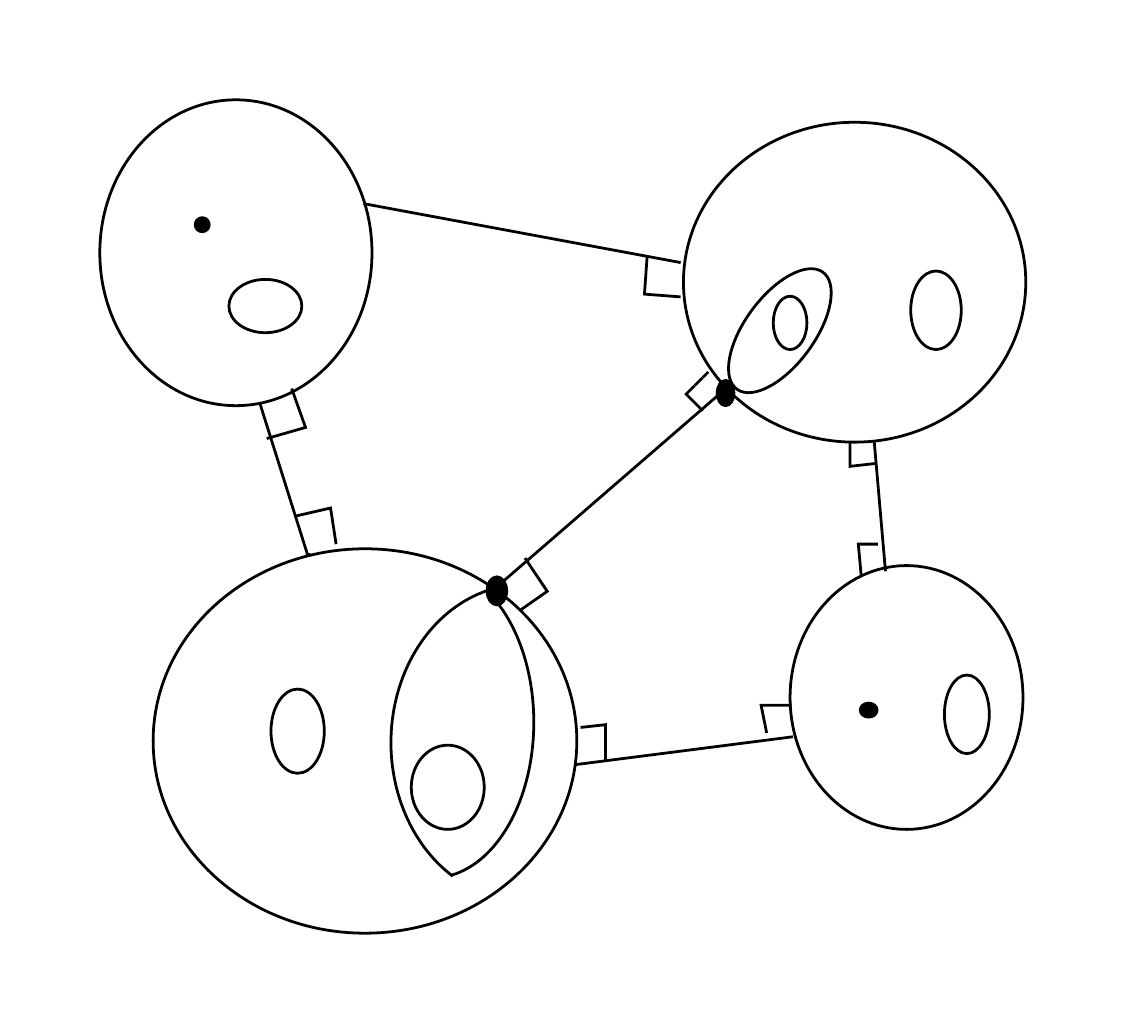,width=7.0cm,angle=0} }}
\vspace{-20pt}
\end{center}
\caption{The curve $\gamma_a$ on $X-S$. }
\end{figure}
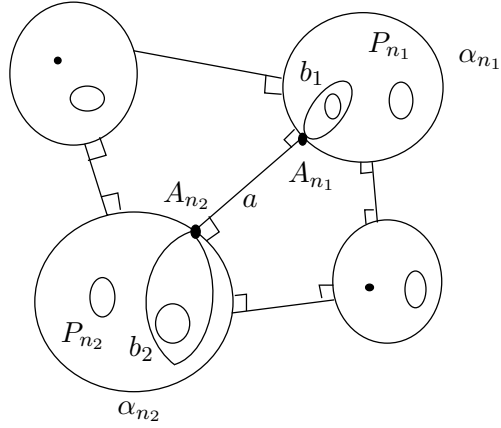

Let $I$ be a geodesic arc of length $1$ in the planar part $X-S$. If $a$ is an orthogonal geodesic arc between boundary geodesics $\alpha_{n_i}$ and $\alpha_{n_j}$ used in our construction, then the points of the corresponding closed geodesic $\gamma_a^{*}$ are on a bounded distance $D$ from both $\alpha_{n_i}$ and $\alpha_{n_j}$ since its length is uniformly bounded. Let us consider the family $\{\gamma_a^{*}\}_a$ and the family of components of $(X-S)-\{\gamma_a^{*}\}_a$ (see Figure 5). There is a number $K$ such that any geodesic arc $I$ of length $1$ can intersect  
at most $K$ of these components. This follows because there is an upper bound on the number of  boundary geodesics of $X-S$ that are within a fixed distance from any point of $X-S$ and the geodesics $\gamma_a^{*}$ stay within fixed distance of the two boundary geodesics they intersect. In addition, there is a uniform bound on the number of geodesics $\gamma_a^{*}$ that meet each boundary geodesic. This implies the existence of a constant $K$ with the above property.  

\begin{figure}[h]
\leavevmode \SetLabels
\endSetLabels
\begin{center}
\AffixLabels{\centerline{\epsfig{file =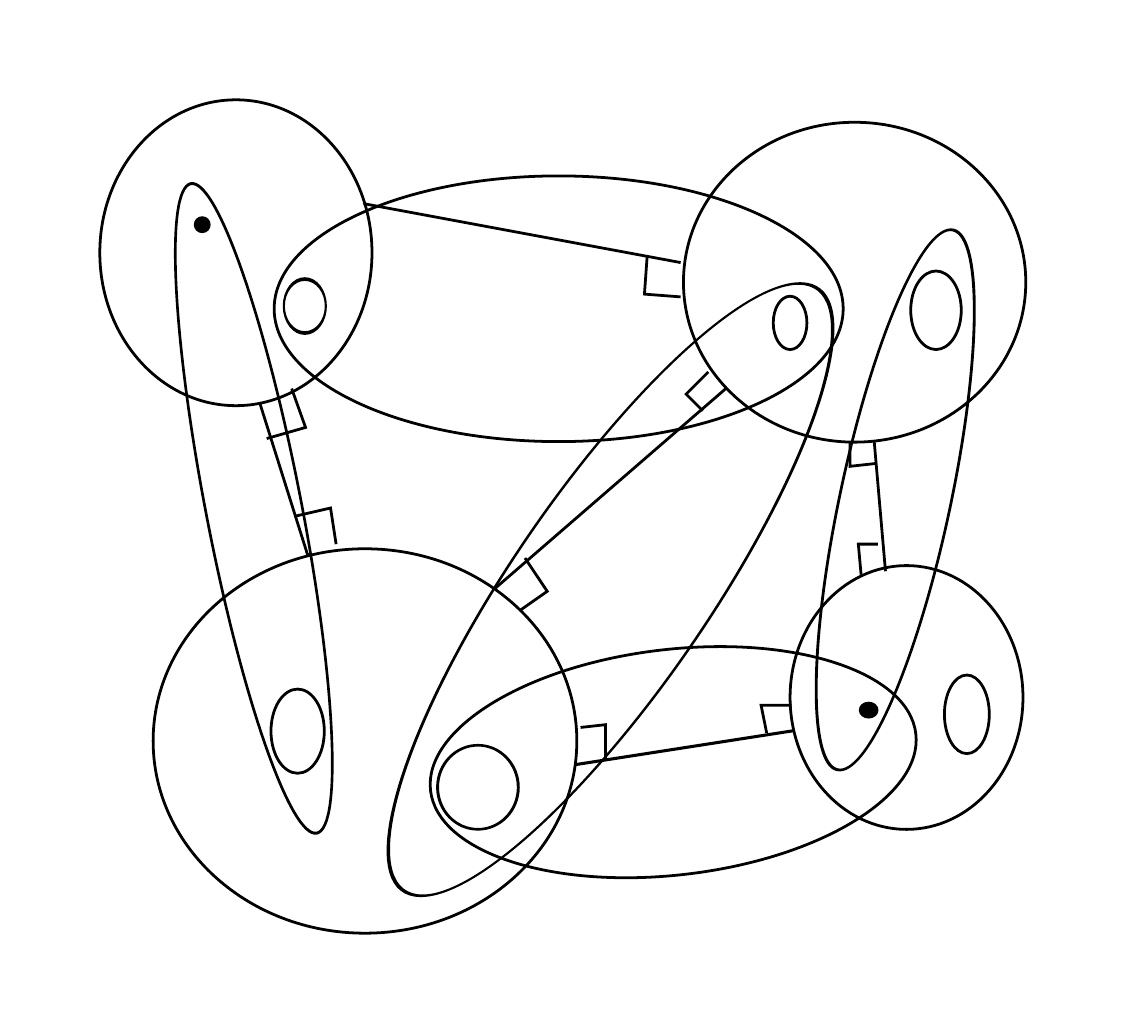,width=7.0cm,angle=0} }}
\vspace{-20pt}
\end{center}
\caption{The complements of $\{\gamma_a^{*}\}_a$ on $X-S$. }
\end{figure}

The boundary of each component of $(X-S)-\{\gamma_a^{*}\}_a$ contains at most $K_1$ arcs on either different boundary geodesics or different geodesics of the family $\{\gamma_a^{*}\}_a$, where $K_1$ is uniform over the whole family because the total number of these geodesics on a definite distance from an arbitrary point is bounded by a uniform constant. 
Since $I$ intersects at most $K$ components $(X-S)-\{\gamma_a^{*}\}_a$, it follows that each geodesic intersecting $I$ intersects at least one and at most $KK_1$ boundary geodesics $\{\alpha_{n_k}\}$ or geodesics of the family $\{\gamma_a^{*}\}_a$. 

Since we established that there exists $C>0$ such that $i(\nu_{\varphi},\alpha_{n_k})\leq C$ for each $k$ and $i(\nu_{\varphi},\gamma_a^{*})\leq C$ for each $a$, it follows that
$$
i(\nu_{\varphi},I)\leq CKK_1
$$
for each $I$ of hyperbolic length $1$ in $X-S$. If $I\subset S$ then a similar estimate follows by the bound on $l_X(\alpha_n)$. Thus we established that $\|\nu_{\varphi}\|_{Th}<\infty$.
\end{proof}

\section{Integrable holomorphic quadratic differentials on surfaces with upper bounded pants decomposition}

Let $X$ be an infinite hyperbolic surface with a geodesic pants decomposition $\mathcal{P}=\{ P_k\}$, where each geodesic pair of pants $P_k$ can have at most two punctures.  Denote by $\{\alpha_n\}$ the set of boundary geodesics of the pants $\{ P_k\}$. 
The pants decomposition $\mathcal{P}=\{ P_k\}$ is {\it upper bounded} if it satisfies the following two conditions: 
\begin{enumerate}
    \item there exists $M>0$ such that 
$$
l_X(\alpha_n)\leq M
$$
for all $n$, 
\item  there exists a subsequence $\{\alpha_{n_k}\}$ such that
$$
l_X(\alpha_{n_k})\to 0
$$
as $k\to\infty$.
\end{enumerate}

If the first condition is satisfied but the second condition is not satisfied then the surface $X$ has bounded pants decomposition.

\begin{thm}
\label{thm:not_bounded}
Let $X$ be a hyperbolic surface equipped with an upper bounded geodesic pants decomposition. 
Then there exists an integrable holomorphic quadratic differential $\varphi_0$ on $X$ such that 
$$
\|\nu_{\varphi_0}\|_{Th}=\infty .
$$
\end{thm}

\begin{proof}
Since $l_X(\alpha_{n_k})\to 0$ as $k\to\infty$, we can choose a subsequence $\{\alpha_{n_{k_i}}\}_i$ such that
$$
l_X(\alpha_{n_{k_i}})\leq \frac{1}{i^2}
$$

The modulus $\mathrm{mod}(R_{n_{k_i}})$ of the standard collar neighborhood $R_{n_{k_i}}$ around $\alpha_{n_{k_i}}$ is asymptotic to $Ci^2$ as $i\to\infty$ (see Maskit \cite{Maskit}). Let $b_i=\sqrt[3]{i}$. Then
$$
\sum_{i=1}^{\infty} b_i^2/\mathrm{mod}(R_{n_{k_i}})<\infty .
$$
By Strebel \cite[Theorem 22.1]{Strebel}, there exists a unique integrable holomorphic quadratic differential $\varphi_0$ on $X$ with closed horizontal trajectories whose cylinders $R_{n_{k_i}}'$ are homotopic to $\alpha_{n_{k_i}}$ and have heights $b_i$. Then
$$
\nu_{\varphi_0}(\alpha_{n_{k_i}})=\sqrt[3]{i}\to\infty
$$
as $i\to\infty$. It follows that $\|\nu_{\varphi_0}\|_{Th}=\infty$ since the atomic measure on $\alpha_{n_{k_i}}$ goes to infinity. 
\end{proof}

The above theorem establishes the existence of an integrable holomorphic quadratic differential $\varphi_0$ on $X$ whose horizontal measured lamination is not Thurston bounded. The construction gives a differential with cylinders homotopic to a sequence of closed geodesics whose lengths go to zero. 
In the following proposition we show that this is one of two possible ways that an integrable holomorphic quadratic differential can have unbounded horizontal measured lamination. 

\begin{prop}
\label{thm:upper_bounded_measured}
Let $X$ be a Riemann surface equipped with an upper bounded geodesic pants decomposition with boundary geodesics $\{\alpha_n\}$. Let $\varphi$ be an integrable holomorphic quadratic differential on $X$. Then
$$
i(\nu_{\varphi},\alpha_n)\leq \frac{C}{\sqrt{l_X(\alpha_n)}}.
$$
\end{prop}

\begin{proof}
Fix a pants decomposition $\mathcal{P}$ with boundary geodesics $\{\alpha_n\}$ such that $l_X(\alpha_n)<M$. In the notation and by the proof of Theorem \ref{thm:integrable_measured_laminations}, we have that
$$
\mathrm{mod}(\Gamma_n)\geq \mathrm{mod}(\mathcal{V}_n)=\int_{T_n'}\frac{1}{l_n(w)}dv
$$
where $\Gamma_n$ is the family of curves connecting the two boundary component of $\mathcal{C}_{\alpha_n}$. By Maskit \cite{Maskit}, we have $\mathrm{mod}(\Gamma_n)\leq \frac{C}{l_X(\alpha_n)}$. Then using the method in the proof of Theorem \ref{thm:integrable_measured_laminations} we obtain
$$
i(\nu_{\varphi},\alpha_n)^2\leq \Big{(}\int_{T_n'}dv\Big{)}^2\leq\int_{T_n'}l_n(w)dv\int_{T_n'}\frac{1}{l_n(w)}dv\leq \frac{C'\|\varphi\|_{L^1(X)}}{l_X(\alpha_n)}.
$$
\end{proof}

\begin{rem}
By the construction in the proof of the above theorem and by the above proposition we conclude that an integrable holomorphic quadratic differential on $X$ gives unbounded horizontal measured foliation if either a sequence of cylinders $R_i$ have unbounded heights or if  $\frac{i(\alpha_{n_i},\nu_{\varphi})}{l_X(\alpha_{n_i})}\to\infty$ for a sequence  $l_X(\alpha_{n_i})\to 0$ as $i\to\infty$. The later condition is not precluded by the upper bound in the above proposition. 
\end{rem}

\section{Finite hyperbolic surfaces with funnels}

In this section $X$ denotes an infinite area hyperbolic surface whose fundamental group is finitely generated. Such surface $X$ necessarily has hyperbolic funnels. We prove

\begin{thm}
\label{thm:funnels}
Let $X$ be a hyperbolic surface with finitely generated fundamental group that contains hyperbolic funnels. If $\varphi\in A(X)$ then
$$
\|\nu_{\varphi}\|_{Th}<\infty .
$$
\end{thm}

\begin{proof}
The convex core of $X$ is a finite surface with geodesic boundary. The complement of the convex core consists of finitely many funnels attached to the boundary of the convex core along simple closed geodesics.

Let $\{\alpha_j\}_{j=1}^k$ be boundary geodesics of a geodesic pants decomposition of the convex core of $X$ with $\alpha_j$ for $j=1,\ldots ,k’$ being boundary geodesics of the convex core.

Since the set $\{\alpha_j\}_{j=1}^k$ is finite, the method of proof of Theorem \ref{thm:integrable_measured_laminations} gives that
$$
\max_j\ i(\alpha_j,\nu_{\varphi})<\infty
$$
and
$$
\max_j\ \nu_{\varphi}(\alpha_j)<\infty .
$$
It follows that the total measure of the geodesics intersecting compact arcs of length one in the convex core of $X$ is bounded away from the infinity. It remains to prove the same statement when the compact arcs are inside the finitely many funnels.

Let $F_j$ be one funnel with boundary geodesic $\alpha_j$. We consider the intersection numbers of geodesic arcs $I$ of length $1$ with the leaves of $\nu_{\varphi}$ which are completely contained in $F_j$. 
Consider a lift $\tilde{F}_j$ of $F_j$ with the boundary geodesic $\tilde{\alpha}_j$. Then $\tilde{F}_j$ is a hyperbolic half-plane with the boundary geodesic $\tilde{\alpha}_j$.
Let $A_j\in\pi_1(X)$ be the primitive hyperbolic element corresponding to $\alpha_j$. Let $\omega_j\subset \tilde{F}_j$ be a fundamental region for the action of $<A_j>$ which is between two geodesics rays $r_1$ and $r_2$ that start and are orthogonal to $\tilde{\alpha}_j$. Then we have $A_j(r_1)=r_2$ up to the exchange of $r_1$ and $r_2$. Since $\tilde{\lambda}_{\tilde{\varphi}}$ is a geodesic lamination invariant under the action of the cyclic group $<A_j>$, it follows that any geodesic that completely stays in $\tilde{F}_j$ cannot intersect more than one of $\{A_j^k(r_1)\}_{k\in\mathbb{Z}}$. This implies that a box of geodesics $Q=[a,b]\times [c,d]$ with $[a,b],[c,d]\subset \partial_{\infty}\tilde{F}_j$ such that at least two geodesic rays from $\{A_j^k(r_1)\}_{k\in\mathbb{Z}}$ separate $[a,b]$ and $[c,d]$ cannot contain geodesics of the support $\tilde{\lambda}_{\tilde{\varphi}}$ of $\tilde{\nu}_{\tilde{\varphi}}$. By elementary hyperbolic geometry, given $M>0$ there is $k>0$ such that any box of geodesics $Q=[a,b]\times [c,d]\subset G(\mathbb{H})$ with the Liouville measure at most $M$ and $\tilde{\lambda}_{\tilde{\varphi}}\cap Q\neq \emptyset$ is contained in an at most $k$ translates of the fundamental set $\omega_j$. 

We need to estimate $\nu_{\varphi}(Q)$. To do so, we consider the leaves of the horizontal foliation $\mathcal{\tilde{V}}_{\tilde{\varphi}}$ whose pairs of endpoints are in $Q$. There exists $k_1>0$ such that this set of leaves can be partitioned into an at most $k_1>0$ strips $\{\mathscr{S}(\beta_i)\}_i$ which injectively project  to strips on $X$. The height of each strip is the $\tilde{\nu}_{\tilde{\varphi}}$ measure of each set of geodesics represented by the strip. If $w=u+iv$ is the natural parameter for $\tilde{\varphi}$ and $T_i$ the transverse set then the argument in the proof of Proposition \ref{thm:upper_bounded_measured} gives
$$
(\int_{T_i}dv)^2\leq (\int_{T_i}l(w)dv)(\int_{T_i}\frac{1}{l(w)}dv)=|\mathscr{S}(\beta_i)|_{\tilde{\varphi}}\cdot \mathrm{mod} \Gamma_{[a,b]\times [c,d]}
$$ 
where $|\mathscr{S}(\beta_i)|_{\tilde{\varphi}}$ is the $\tilde{\varphi}$-area of the strip and $\Gamma_{[a,b]\times [c,d]}$ is the family of curves in $\Delta$ that connects $[a,b]$ with $[c,d]$. Since the strip $\mathscr{S}(\beta_i)$ maps injectively to $X$ it follows that $|\mathscr{S}(\beta_i)|_{\tilde{\varphi}}\leq \int_X|\varphi (w)|dudv$. Moreover, there exists $C>0$ such that $\mathrm{mod} \Gamma_{[a,b]\times [c,d]}\leq C$. Thus
$$
\sum_i\int_{T_i}dv\leq C'
$$
for all $i$, where $C'=k_1C\int_X|\varphi (w)|dudv$. Thus
$$
i(\nu_{\varphi},I)\leq C'
$$
for all hyperbolic arcs $I$ of length $1$ in $F_j$ and the theorem is proved. 
\end{proof}

\section{Non-integrable holomorphic quadratic differentials on infinite surfaces}

In this section we assume that $X$ is an infinite hyperbolic surface such that the group $\pi_1(X)$ is of the first kind. In other words, the surface $X$ is formed by gluing countably many geodesic pairs of pants and no funnels or half-planes are added. Let $\varphi$ be a non-integrable holomorphic quadratic differential on $X$. Recall that by the work of  Marden and Strebel \cite{MardenStrebel}, \cite{Strebel} one can still define the horizontal measured lamination $\nu_{\varphi}$ on $X$. In the following theorem we give sufficient conditions on the geometry of the surface and (non-integrable) holomorphic quadratic differential to guarantee that the induced horizontal measured lamination is bounded.

\begin{thm}
\label{thm:nonint}
Let $X$ be an infinite hyperbolic surface equipped with two bounded geodesic pants decomposition $\mathcal{P}=\{ P_k\}$ and $\mathcal{P}'=\{ P_k’\}$ which do not share a boundary geodesic. Let $\{\alpha_n\}$, $\{\alpha_n'\}$ be the geodesic boundaries of $P_k$ and $P_k'$. Let $\mathcal{C}_n$ and $\mathcal{C}_n'$ be the standard collars of $\alpha_n$ and $\alpha_n'$. If 
$$
\sup_n\int_{\mathcal{C}_n}|\varphi (\zeta )|d\xi d\eta<\infty 
$$
and
$$
\sup_n\int_{\mathcal{C}_n'}|\varphi (\zeta )|d\xi d\eta<\infty 
$$
then 
$$
\|\nu_{\varphi}\|_{Th}<\infty .
$$
\end{thm}

\begin{proof}
Consider the family $\mathcal{V}_n$ of subleaves of the vertical foliation $\mathcal{V}_{\varphi}$ of $\varphi$ that connect one boundary of $\mathcal{C}_n$ to the other. Let $T_n$ be the transverse set to $\mathcal{V}_n$ and $T_n'$ its image in the natural parameter $w=u+iv$ of $\varphi$. Then
$$
B:=\sup_n\int_{\mathcal{C}_n}|\varphi (z)|dxdy\geq \sup_n\int_{T_n'}l_n(w)dv.
$$
By the same method of proving (\ref{eq:intersection_cuff}), we get
$$
\sup_ni(\nu_{\varphi},\alpha_n)^2\leq \Big{[}\sup_n\int_{T_n'}l_n(w)dv\Big{]}\Big{[}\sup_n\int_{T_n'}\frac{1}{l_n(w)}dv\Big{]}<\infty .$$
Indeed, $\int_{T_n'}\frac{1}{l_n(w)}dv=\mathrm{mod}(\Gamma_n)\leq\mathrm{mod}(\mathcal{C}_n)\leq B'$ for some $B'$ independent of $n$ by the bound on $l_X(\alpha_n)$. 

The same method shows that $\sup_ni(\nu_{\varphi},\alpha_n')<\infty$. Since
$$
\nu_{\varphi}(\alpha_n)\leq i(\nu_{\varphi}, \alpha_n')
$$
it follows that
$$
\sup_n \nu_{\varphi}(\alpha_n)<\infty .
$$

This implies $\|\nu_{\varphi}\|_{Th}<\infty$ as in the proof of Theorem \ref{thm:integrable_measured_laminations}.
\end{proof}

\author{Dragomir \v Sari\' c, Department of Mathematics,  Graduate Center and Queens College, CUNY, Dragomir.Saric@qc.cuny.edu}

\end{document}